\documentclass[a4paper,UKenglish]{scrartcl}

\pdfoutput=1

\newcommand{\cslonly}[1]{}
\newcommand{\arxivonly}[1]{#1}
\newcommand{\cslorarxiv}[2]{#2}
\newcommand{\journalonly}[1]{}

\usepackage{microtype}
\arxivonly{
  \usepackage{authblk}

  \usepackage{amsmath,amsthm,amssymb}

  \usepackage[dvipsnames]{xcolor}
  \definecolor{darkgreen}{rgb}{0,0.45,0}
  \definecolor{darkred}{rgb}{0.75,0,0}
  \definecolor{darkblue}{rgb}{0,0,0.6}
  \usepackage[colorlinks,citecolor=darkgreen,linkcolor=darkred,urlcolor=darkblue]{hyperref}
  \usepackage{breakurl}

  \theoremstyle{plain}
  \newtheorem{theorem}{Theorem}[subsection]
  \newtheorem{lemma}[theorem]{Lemma}
  \newtheorem{corollary}[theorem]{Corollary}
  \theoremstyle{definition}
  \newtheorem{definition}[theorem]{Definition}
  \newtheorem{example}[theorem]{Example}
  \theoremstyle{remark}
  \newtheorem*{remark}{Remark}

  \usepackage[osf]{libertine}
  \usepackage[scaled=0.96]{zi4}  

  \usepackage{listings}
}

\theoremstyle{plain}
\newtheorem*{theorem*}{Theorem}

\usepackage{mathtools} 
\DeclareMathOperator{\dif}{d \!}
\usepackage{todonotes}
\usepackage{tikz-cd}
\usepackage{hyphenat} 

\newcommand{\sm}[1]{\Sigma{(#1).}}
\newcommand{\dpt}[1]{\Pi{(#1).}}
\newcommand{\ex}[1]{\exists{(#1).}}
\newcommand{\fa}[1]{\forall{(#1).}}

\newcommand{\isProp}{\operatorname{isHProp}}
\newcommand{\gL}{\operatorname{locatesRight}}
\newcommand{\gU}{\operatorname{locatesLeft}}

\newcommand{\proj}[1]{{\operatorname{pr}_{#1}}}
\newcommand{\fst}{\proj1}

\newcommand{\Biimp}{\Leftrightarrow}
\newcommand{\Implies}{\Rightarrow}
\newcommand{\Prop}{{\operatorname{HProp}}}
\newcommand{\DProp}{{\operatorname{DHProp}}}

\newcommand{\set}[1]{\left\{ \, #1 \,\right\}}
\newcommand{\trunc}[1]{\left\| #1 \right\|}
\newcommand{\truncm}[1]{\left| #1 \right|}
\newcommand{\abs}[1]{\left| #1 \right|}

\newcommand{\floor}[1]{\left\lfloor #1 \right\rfloor}
\newcommand{\minim}{\operatorname{minimal}}
\newcommand{\UU}{\mathcal{U}}

\newcommand{\Empty}{\mathbf{0}}
\newcommand{\unit}{\mathbf{1}}
\newcommand{\bool}{\mathbf{2}}
\newcommand{\true}{\mathsf{t\!t}} 
\newcommand{\false}{\mathsf{f\!f}} 
\newcommand{\inl}{{\operatorname{inl}}}
\newcommand{\inr}{{\operatorname{inr}}}

\newcommand{\N}{\mathbb{N}}
\newcommand{\Z}{\mathbb{Z}}
\newcommand{\Q}{\mathbb{Q}}
\newcommand{\R}{\mathbb{R}}
\newcommand{\RC}{{\R_\mathbf{C}}}
\newcommand{\RD}{{\R_\mathbf{D}}}

\newcommand{\SL}{\operatorname{locator}}
\newcommand{\SLR}{\mathfrak{L}}
\newcommand{\RDL}{{\R_\mathbf{D}^\SLR}}
\newcommand{\apart}{\mathrel{\#}}
\newcommand{\eqv}{\simeq}
\newcommand{\Iff}{\Leftrightarrow}
\newcommand{\pow}{\mathcal{P}}
\newcommand{\isCut}{\operatorname{isDedekindCut}}
\newcommand{\boundedLower}{\operatorname{boundedLower}}
\newcommand{\boundedUpper}{\operatorname{boundedUpper}}
\newcommand{\closedLower}{\operatorname{closedLower}}
\newcommand{\closedUpper}{\operatorname{closedUpper}}
\newcommand{\openLower}{\operatorname{openLower}}
\newcommand{\openUpper}{\operatorname{openUpper}}
\newcommand{\transitive}{\operatorname{transitive}}
\newcommand{\located}{\operatorname{located}}
\newcommand{\cutStruct}{\isCut^\S}
\newcommand{\bounderLower}{\boundedLower^\S}
\newcommand{\bounderUpper}{\boundedUpper^\S}
\newcommand{\closerLower}{\closedLower^\S}
\newcommand{\closerUpper}{\closedUpper^\S}
\newcommand{\openerLower}{\openLower^\S}
\newcommand{\openerUpper}{\openUpper^\S}
\newcommand{\transitor}{\transitive^\S}
\newcommand{\locator}{\located^\S}
\newcommand{\PEM}{\mathrm{PEM}} 
\newcommand{\WLPO}{\mathrm{WLPO}} 

\title{Extensional constructive real analysis\\ via locators}

\cslorarxiv{

\author{Auke B. Booij}{School of Computer Science, Universiy of
  Birmingham\\{Birmingham, UK}}{}{}{}
\authorrunning{A.B. Booij}
\Copyright{Auke B. Booij}

\subjclass{}


\EventEditors{John Q. Open and Joan R. Access}
\EventNoEds{2}
\EventLongTitle{42nd Conference on Very Important Topics (CVIT 2016)}
\EventShortTitle{CVIT 2016}
\EventAcronym{CVIT}
\EventYear{2016}
\EventDate{December 24--27, 2016}
\EventLocation{Little Whinging, United Kingdom}
\EventLogo{}
\SeriesVolume{42}
\ArticleNo{23}
}{
  \author{Auke B. Booij}
  \affil{School of Computer Science, Universiy of
  Birmingham\\{Birmingham, UK}}

}

\begin{document}

\maketitle

\begin{abstract}
  Real numbers do not admit an extensional procedure for observing
  discrete information, such as the first digit of its decimal
  expansion, because every extensional, computable map from the reals
  to the integers is constant, as is well known.  We overcome this by
  considering real numbers equipped with additional structure, which
  we call a locator.  With this structure, it is possible, for
  instance, to construct a signed-digit representation or a Cauchy
  sequence, and conversely these intensional representations give rise
  to a locator.  Although the constructions are reminiscent of
  computable analysis, instead of working with a notion of
  computability, we simply work constructively to extract observable
  information, and instead of working with representations, we
  consider a certain locatedness structure on real numbers.
\end{abstract}

\section{Introduction}
\label{sec:intro}

It is well known how to compute with real numbers intensionally, with
equality of real numbers specified by an imposed equivalence relation
on
representations~\cite{bishop:constructive,lcf:analysis,oconnor:formalizing},
such as Cauchy sequences or streams of digits.  It has to be checked
explicitly that functions on the representations preserve such
equivalence relations.  Discrete observations, such as finite decimal
approximations, can be made because representations are given, but a
different representation of the same real number can result in a
different observation, and hence discrete observations are necessarily
non-extensional.

In univalent mathematics, equality of real numbers can be captured by
identity types directly, rather than by an imposed equivalence
relation, thus avoiding the use of setoids.  Preservation of equality
of real numbers is automatic, but the drawback is that we are
prevented from making any discrete observations of arbitrary real
numbers.
This kind of problem is already identified by
Hofmann~\cite[Section~5.1.7.1]{hofmann:extensional} for an extensional
type theory.  Discrete observations of real numbers are made by
breaking extensionality using a \emph{choice operator}, which does not
give rise to a \emph{function}.

\medskip


To avoid breaking extensionality, the central idea of this paper is to
restrict our attention to real numbers that can be equipped with a
simple structure called a \emph{locator}.
Such a locator is a strengthening of the locatedness property of
Dedekind cuts.  While the locatedness of a real number $x$ says that
for rational numbers $q<r$ we have the property $q<x$ or $x<r$, a
locator produces a specific selection of one of $q<x$ and $x<r$.  In
particular, the same real number can have different locators, and it
is in this sense that locators are structure rather than property.

In a constructive setting such as ours, not all real numbers have
locators, and we prove that the ones that do are the ones that have
Cauchy representations in Section~\ref{sec:sdr}.  However, working
with locators rather than Cauchy representations gives a development
which is closer to that of traditional real analysis.  For example, we
can prove that if $x$ has a locator, then so does $e^x$, and this
allows to compute $e^x$ when working constructively, so that we say
that the exponential function \emph{lifts to locators}.  As another
example, if $f$ is given a modulus of continuity and lifts to locators,
then $\int_0^1f(x)\dif x$ has a locator and we can compute the
integral in this way.

Thus the difference between locatedness and locators is that one is
property and the other is structure.
Plain Martin-L\"of type theory is not enough to capture this
distinction because, for example, it allows to define the notion of
locator as structure but not the notion of locatedness as property,
and therefore it does not allow to define the type of Dedekind reals
we have in mind, whose identity type should capture directly the
intended notion of equality of real numbers.
A good foundational system to account for such distinctions is
univalent type theory (UTT), also known as homotopy type
theory~\cite{hottbook}.  For us, it is enough to work in the fragment
consisting of Martin-L\"of type theory with propositional truncation,
propositional extensionality and function extensionality (see
Section~\ref{sec:prelim}).  The need for univalence would arise only
when considering types of sets with structure such as the type of
metric spaces or the type of Banach spaces for the purposes of
functional analysis.

We believe that our constructions can also be carried out in other
constructive foundations such as CZF, the internal language of an
elementary topos with a natural numbers object, or Heyting arithmetic
of finite types.  Our choice of UTT is to some extent a practical one,
as it is a constructive system with sufficient extensionality, which
admits, at least in theory, applications in proof assistants allowing
for computation using the techniques in this paper.

\medskip

In summary, the work has two aspects.  One aspect is that instead of
working with functions on intensional representations, we work with
functions on real numbers that lift representations.  The second
aspect is the particular representation that seems suitable.

We describe the assumptions on the foundational system in
Section~\ref{sec:prelim}.

The definition and basic theory of locators is given in
Section~\ref{sec:locators}.  We construct locators for rationals in
Section~\ref{sec:rationals}.  We discuss preliminaries for observing
data from locators in Sections~\ref{sec:logic-locators}
and~\ref{sec:bounded-search}, which is then used to compute rational
bounds in Section~\ref{sec:computing-bounds}.  We compute locators for
algebraic operations in Sections~\ref{sec:algebraic-operations} and
for limits in Section~\ref{sec:locators:limits}.  We compute signed
digit representations for reals with locators in
Section~\ref{sec:sdr}.  Given a real and a locator, we strengthen the
properties for being a Dedekind cut into structure in
Section~\ref{sec:cuts-as-structure}.

We show some ways of using locators in constructive analysis in
Section~\ref{sec:analysis}.  We compute locators for
integrals in Section~\ref{sec:integrals}.  We discuss how locators can
help computing roots of functions in Section~\ref{sec:ivt}.

\section{Preliminaries}
\label{sec:prelim}

We work in type theory with universes $\UU$ and $\UU'$ with
$\UU:\UU'$, identity types $x=_X y$ for $x,y:X$, a unit type $\unit$,
an empty type $\Empty$, a natural numbers type $\N$, dependent sum
types $\Sigma$, dependent product types $\Pi$ and propositional
truncation $\trunc{\,\cdot\,}$ (see Section~\ref{sec:props}).  We assume
function extensionality, which can be stated as the claim that all
pointwise equal functions are equal.  We assume propositional
extensionality, namely the claim that if $P$ and $Q$ are propositions
in the sense of Section~\ref{sec:props}, and $P\Implies Q$ and
$Q\Implies P$, then $P=Q$.

\subsection{Propositions}
\label{sec:props}

\begin{definition}\label{def:prop}
  A \emph{proposition} is a type $P$ all whose elements are equal, which
  is expressed type-theoretically as
  \[
    \isProp(P)\coloneqq\dpt{p,q:P}(p=_P q).
  \]
  We have the type
  $\Prop\coloneqq\sm{P:\UU}\isProp(P)$ of all propositions, and we
  conflate elements of $\Prop$ with their underlying type, that is,
  their first projection.
\end{definition}

We assume that every type has a propositional truncation.

\begin{definition}\label{def:truncation}
  The \emph{propositional truncation} $\trunc{X}$ of a type $X$ is a
  proposition together with a \emph{truncation map}
  $\truncm{\,\cdot\,}:X\to \trunc{X}$ such that for any other
  proposition $Q$, given a map $g:X\to Q$, we obtain a map
  $h:\trunc{X}\to Q$.
\end{definition}
\begin{remark}
  The uniqueness of the obtained map $\trunc{X}\to Q$ follows from the
  fact that $Q$ is a proposition, and function extensionality.
\end{remark}
We may also think of propositional truncations categorically, in which
case they have the universal property that given a map $X\to Q$ as in
the diagram below, we obtain the vertical map, which automatically
makes the diagram commute because $Q$ is a proposition, and which is
automatically equal to any other map that fits in the diagram.
\begin{center}
  \begin{tikzcd}
    X \arrow[r,"\truncm{\,\cdot\,}"] \arrow[rd] & \trunc{X}
    \arrow[d,dashrightarrow]
    \\
    & Q
  \end{tikzcd}
\end{center}

Propositional truncations can be defined as higher-inductive types, or
constructed via impredicative encodings assuming propositional
resizing.

Even though the elimination rule in Definition~\ref{def:truncation}
only constructs maps into propositions, we can \emph{sometimes} get a map
$\trunc{X}\to X$, as we discuss in Theorem~\ref{thm:nat-dec}.

\begin{definition}\label{def:props:interp}
  Truncated logic is defined by the following, where $P,Q:\Prop$
  and~$R:X\to\Prop$~\cite[Definition~3.7.1]{hottbook}:
  \begin{align*}
    \top            & \coloneqq  \unit \\
    \bot            & \coloneqq  \Empty \displaybreak[1] \\
    P \land Q       & \coloneqq  P \times Q \\
    P \Implies Q & \coloneqq  P \to Q \\
    P \Biimp Q & \coloneqq  P=Q \\
    \neg P          & \coloneqq  P \to \Empty \\
    P \lor Q        & \coloneqq  \trunc{P + Q} \displaybreak[1] \\
    \fa{x : X} R(x) & \coloneqq  \dpt{x : X} R(x) \\
    \ex{x : X} R(x) & \coloneqq \trunc{\sm{x : X} R(x)}
  \end{align*}
\end{definition}

We use the following terminological conventions throughout the work.
\begin{definition}\label{def:prop:terminology}
  We refer to types that are propositions as \emph{properties}.  We
  refer to types which may have several inhabitants as \emph{data} or
  \emph{structures}.  We indicate the use of truncations with the verb
  ``to \emph{exist}'': so the claim ``there exists an $A$ satisfying
  $B$'' is to be interpreted as $\ex{a:A}B(a)$, and ``there exists an
  element of X'' is to be interpreted as $\trunc{X}$.  Most other
  verbs, including ``to have'', ``to find'', ``to construct'', ``to
  obtain'', ``to get'', ``to give'', ``to equip'', ``to yield'' and
  ``to compute'', indicate the absence of truncations.
\end{definition}
\begin{example}
  One attempt to define when $f:X\to Y$ is a surjection is
  \[
    \dpt{y:Y}\sm{x:X}fx=y.
  \]
  In fact, this is rather called split surjective, as from that
  structure, we obtain a map $Y\to X$ which is inverse to $f$: so we
  have defined when a function is a \emph{section}.  Rather defining
  surjectivity as
  \[
    \fa{y:Y}\ex{x:X}fx=y,
  \]
  by virtue of using the \emph{property} $\ex{x:X}fx=y$, does not
  yield an inverse map.

  In words, we say that $f$ is a surjection if for every $y:Y$ there
  \emph{exists} a pre-image.  The terminology that every $y:Y$
  \emph{has} a pre-image means a choice of pre-images, which
  formalizes sections.
\end{example}

\begin{example}
  Given a function $f:A\to B$, the \emph{image} of $f$ is the
  collection of elements $b:B$ that are reached by $f$, that is, for
  which there is an element $a:A$ such that $fa=_Bb$.  The
  propositions-as-types interpretation would formalize this as
  \[
    \sm{b:B}\sm{a:A}fa=_Bb.
  \]
  However, because the type $\sm{b:B}fa=_Bb$ is
  contractible~\cite[Lemma~3.11.8]{hottbook}, in fact this type is
  equivalent to the type $A$ itself, in the sense that there is a map
  with a left pointwise inverse and a right pointwise inverse, and so
  it does not adequately represent the image of $f$.

  Using truncations, we instead formalize the image of $f$ as the
  collection of elements of $B$ for which there \emph{exists} a
  pre-image along $f$, that is, in UTT the image of $f$ is formalized
  as:
  \[
    \sm{b:B}\ex{a:A}fa=_Bb,
  \]
  noting that the inner $\Sigma$ is truncated whereas the outer is
  not: we want to distinguish elements in the image of $f$, but we do
  not want to distinguish those elements based on a choice of
  pre-image in $A$.
\end{example}

\begin{example}
  We may compute the integral of a uniformly continuous function $f$
  as:
  \[
    \int_a^bf(x)\dif x=
    \lim_{n\to\infty}
    \frac{b-a}{n}
    \sum_{k=0}^{n-1}
    f\left(a+k\cdot\frac{b-a}{n}\right)
    .
  \]
  The construction of the limit value, e.g.\ as in
  Lemma~\ref{lem:dedekind:lim}, uses the modulus of uniform continuity
  of $f$ as in Definition~\ref{def:uniform:cty}.  However, since the
  integral is independent of the choice of modulus, by unique choice,
  e.g.\ as in Theorem~5.4 of~\cite{DBLP:journals/corr/KrausECA16}, the
  \emph{existence} (defined constructively as in
  Definition~\ref{def:props:interp}) of a modulus of uniform
  continuity suffices to compute the integral.  We discuss this
  further in Sections~\ref{sec:locators:limits}
  and~\ref{sec:integrals}.
\end{example}

\subsection{Dedekind reals}
\label{sec:dedekind-reals}

Although the technique of equipping numbers with locators can be
applied to any archimedean ordered field, for clarity and brevity we
will work with the Dedekind reals $\RD$ as defined in The Univalent
Foundations Program~\cite{hottbook}.  A more general description is
given in Booij~\cite{booij:thesis}.

\begin{definition}
  A \emph{predicate} $B$ on a type $X:\UU$ is a map $B:X\to\Prop$.
  For $x:X$ we write $(x\in B)\coloneqq B(x)$.
\end{definition}

A Dedekind real is defined by a pair $(L,U)$ of predicates on $\Q$
with some properties.  To phrase these properties succinctly, we use
the following notation for $x=(L,U)$:
\begin{align*}
  (q<x) &\coloneqq (q \in L) \qquad\text{and} \\
  (x<r) &\coloneqq (r \in U).
\end{align*}
This is justified by the fact that $q\in L$ holds iff $i(q)<x$, with
$i:\Q\hookrightarrow\RD$ the canonical inclusion of the rationals into the
Dedekind reals.

\begin{definition}\label{def:dedekind}
  A pair $x=(L,U)$ of predicates on the rationals is a \emph{Dedekind
    cut} or \emph{Dedekind real} if it satisfies the four Dedekind
  properties:
  \begin{enumerate}
  \item \emph{bounded:} $\ex{q : \Q} q<x$ and
    $\ex{r : \Q} x<r$.
  \item \emph{rounded:} For all $q,r : \Q$,
    \begin{align*}
      q<x & \Biimp \ex{q' : \Q} (q < q') \land (q'<x)
            \qquad\text{and}
      \\
      x<r & \Biimp \ex{r' : \Q} (r' < r) \land (x<r').
    \end{align*}
  \item \emph{transitive:} $(q<x)\land (x<r)\Implies (q<r)$ for all
    $q, r : \Q$.
  \item \emph{located:} $(q < r) \Implies (q<x) \lor (x<r)$ for all
    $q, r : \Q$.
  \end{enumerate}

  The collection $\RD:\UU'$ of pairs of predicates $(L,U)$ together
  with proofs of the four properties, collected in a $\Sigma$-type, is
  called the \emph{Dedekind reals}.
\end{definition}
\begin{remark}
  The Univalent Foundations Program~\cite{hottbook} has
  \emph{disjointness}
  \[
    \fa{q:\Q}\neg(x<q\land q<x)
  \]
  instead of the transitivity property, which is equivalent to it in
  the presence of the other conditions, and it is this disjointedness
  condition that we use most often in proofs.
\end{remark}
\begin{proof}
  Assuming transitivity, if $x<q\land q<x$, then transitivity yields
  $q<q$, which contradicts irreflexivity of $<$ on the rationals,
  which shows disjointedness.

  Conversely, if $q<x$ and $x<r$, apply trichotomy of the rationals on
  $q$ and $r$: in case that $q<r$ we are done, and in the other two
  cases we obtain $x<q$, contradicting disjointness.
\end{proof}
\begin{definition}
  For Dedekind reals $x$ and $y$, we define the strict ordering
  relation by
  \[
    x < y \coloneqq \ex{q:\Q}x<q<y
  \]
  where $x<q<y$ means $(x<q)\land(q<y)$, and their \emph{apartness} by
  \[
    x\apart y \coloneqq(x<y)\lor(y<x).
  \]
  As is typical in constructive analysis, we have
  $x\apart y\Implies \neg(x=y)$, but not the converse.
\end{definition}
The following proof that $\RD$ is Cauchy complete is based on
The Univalent Foundations Program~\cite[Theorem~11.2.12]{hottbook}.
\begin{lemma}\label{lem:dedekind:lim}
  The Dedekind reals are Cauchy complete.  More explicitly, given a
  modulus of Cauchy convergence $M$ for a sequence $x$ of Dedekind
  reals, i.e.\ a map $M:\Q_+\to\N$ such that
  \[
    \fa{\varepsilon:\Q_+}\fa{m,n:\N}m,n\geq
    M(\varepsilon)\Implies\abs{x_m-x_n}<\varepsilon,
  \]
  we can compute $l:\RD$ as the Dedekind cut defined by:
  \begin{align*}
    (q<l)
    & \coloneqq
      \ex{\varepsilon,\theta:\Q_+}
      (q+\varepsilon+\theta<x_{M(\varepsilon)}) ,\\
    (l<r)
    & \coloneqq
      \ex{\varepsilon,\theta:\Q_+}
      (x_{M(\varepsilon)}<r-\varepsilon-\theta) ,
  \end{align*}
  and $l$ is the limit of $x$ in the usual sense:
  \[
    \fa{\varepsilon:\Q_+}\ex{N:\N}\fa{n:\N}n\geq N\Implies\abs{x_n-l}<
    \varepsilon.
  \]
\end{lemma}
\begin{proof}
  Inhabitedness and roundedness of $l$ are straightforward.  For
  transitivity, suppose $q<l<r$, then we wish to show $q<r$.
  There exist $\varepsilon,\theta,\varepsilon',\theta':\Q_+$ with
  $q+\varepsilon+\theta<x_{M(\varepsilon)}$ and
  $x_{M(\varepsilon')}<r-\varepsilon'-\theta'$.  Now
  $\abs{x_{M(\varepsilon)}-x_{M(\varepsilon')}}\leq\max(\varepsilon,\varepsilon')$,
  so either $q+\theta<x_{M(\varepsilon')}$ or
  $x_{M(\varepsilon)}<r-\theta$, and in either case $q<r$.

  For locatedness, suppose $q<r$.  Set
  $\varepsilon\coloneqq\frac{r-q}{5}$, so that
  $q+2\varepsilon<r-2\varepsilon$.  By locatedness of $x_\varepsilon$,
  we have
  $(q+2\varepsilon<x_\varepsilon)\lor(x_\varepsilon<r-2\varepsilon)$,
  hence $(q<l)\lor(l<r)$.

  In order to show convergence, let $\varepsilon:\Q_+$, set
  $N\coloneqq M(\varepsilon)$, and let $n\geq N$.  We need to show
  $\abs{x_n-l}\leq\varepsilon$, or equivalently,
  $-\varepsilon\leq x_n-l\leq\varepsilon$.  For
  $x_n-l\leq\varepsilon$, suppose that $\varepsilon<x_n-l$, or
  equivalently, $l<x_n-\varepsilon$.  There exist
  $\varepsilon',\theta':\Q_+$ with
  $x_{M(\varepsilon')}<x_n-\varepsilon-\varepsilon'-\theta'$, or
  equivalently,
  $\varepsilon+\varepsilon'+\theta'<x_n-x_{M(\varepsilon')}$, which
  contradicts $M$ being a modulus of Cauchy convergence.  We can
  similarly show $-\varepsilon\leq x_n-l$.
\end{proof}
We denote limits of sequences by $\lim_{n\to\infty}x_n$.
\begin{example}[Exponential function]\label{ex:exp}
  We can define the exponential function $\exp:\RD\to\RD$ as
  $\exp(x)=\sum_{k=0}^\infty\frac{x^k}{k!}$.  We obtain the
  \emph{existence} of a modulus of Cauchy convergence by boundedness
  (as in Definition~\ref{def:dedekind}) of $x$.
\end{example}

\section{Locators}
\label{sec:locators}

The basic idea is that we equip real numbers with the structure of a
\emph{locator}, defined in Section~\ref{sec:locator:definition}.  The
purpose of the work is to show \emph{how} to extract discrete
information from an existing theory of real analysis in UTT.

The following example, which will be fully proved in
Theorem~\ref{thm:exact:ivt}, illustrates how we are going to use
locators.  Suppose $f$ is a pointwise continuous function, and $a<b$
are real numbers with locators.  Further suppose that $f$ is locally
nonconstant, that $f(x)$ has a locator whenever $x$ has a locator, and
that $f(a)\leq0\leq f(b)$.  Then we can find a root of $f$, which
comes equipped with a locator.  For the moment, we provide a proof
sketch, to motivate the techniques that we are going to develop in
this section.  We define sequences $a,b:\N\to\RD$ with
$a_n<a_{n+1}<b_{n+1}<b_n$, with $f(a_n)\leq0\leq f(b_n)$, with
$b_n-a_n\leq(b-a)\left(\frac{2}{3}\right)^n$, and such that all $a_n$
and $b_n$ have locators.  Set $a_0=a$, $b_0=b$.  Suppose $a_n$ and
$b_n$ are defined.  We will explain in the complete proof of
Theorem~\ref{thm:exact:ivt} how to to find $q_n$ with
$\frac{2a_n+b_n}{3}<q_n<\frac{a_n+2b_n}{3}$ and $f(q_n)\apart 0$.  The
important point for the moment, is that this is possible precisely
because we have locators.
\begin{itemize}
\item If $f(q_n)>0$, then set $a_{n+1}\coloneqq a_n$ and
  $b_{n+1}\coloneqq q_n$.
\item If $f(q_n)<0$, then set $a_{n+1}\coloneqq q_n$ and
  $b_{n+1}\coloneqq b_n$.
\end{itemize}
The sequences converge to a number $x$.  For any positive rational
$\varepsilon$, we have $\abs{f(x)}\leq\varepsilon$, hence $f(x)=0$.
This completes our sketch.

\medskip

We need to explain why the sequences $a$ and $b$ come equipped with
locators, and why their limit $x$ has a locator.  In fact, all $q_n$
are rationals, and hence have locators, as discussed in
Section~\ref{sec:rationals}.  The number $q_n$ is constructed using
the central techniques for observing data from locators, see
Sections~\ref{sec:logic-locators} and~\ref{sec:bounded-search}.  These
techniques can then also be used in Section~\ref{sec:computing-bounds}
to compute rational bounds.  Locators for $\frac{2a_n+b_n}{3}$ and
$\frac{a_n+2b_n}{3}$ can be constructed as locators for algebraic
operations, as in Section~\ref{sec:algebraic-operations}.  Locators
for limits are discussed in Section~\ref{sec:locators:limits}.

We compute signed digit representations for reals with locators in
Section~\ref{sec:sdr}.  Given a real and a locator, we strengthen the
properties for being a Dedekind cut into structure in
Section~\ref{sec:cuts-as-structure}.

\subsection{Definition}
\label{sec:locator:definition}

Recall that there is a canonical embedding of the rationals into
$\RD$.  Throughout the remainder of this paper we identify $q:\Q$ with
its embedding $i(q):\RD$.

Recall from Definition~\ref{def:dedekind} that a pair of predicates on
the rationals $x=(L,U)$ is \emph{located} if
$\fa{q,r:\Q}(q < r) \Implies (q<x) \lor (x<r)$.  Indeed, this property
holds for an arbitrary $x:\RD$ by cotransitivity of $<$.
\begin{definition}\label{def:loc}
  A \emph{locator} for $x:\RD$ is a function
  $\ell:\dpt{q,r:\Q}q<r\to(q<x)+(x<r)$.  We denote by $\SL(x)$ the
  type of locators on $x$.  That is, we replace the logical
  disjunction in locatedness by a disjoint sum, so that we get
  structure rather than property, allowing us to compute.
\end{definition}

A locator for $x$ can be thought of as falling in the Dedekind tradition of
considering the rationals to the left and right of $x$, in contrast with
Cauchy-style representations such as sequences of nested intervals.  Whereas
existing Dedekind-style developments directly define a fixed notion of real
number~\cite{bridges:vita}, locators are a structure that can be defined for an
arbitrary type of reals.

A locator can be seen as an analogue to a Turing machine representing
a computable real number, in the sense that it will provide us with
enough data to be able to type-theoretically compute, for instance,
signed-digit expansions.  However, a locator does not express that a
given real is a computable real: in the presence of excluded middle,
there exists a locator for every $x:\RD$, despite not every real being
computable.  To make this precise, we first formalize the principle of
excluded middle type-theoretically.
\begin{definition}\label{def:decidable}
  A \emph{decidable} proposition is a proposition $P$
  such that $P+\neg P$.  We have the collection
  \[
    \DProp\coloneqq\sm{P:\Prop}P+\neg P
  \]
  of decidable propositions.  We identify elements of $\DProp$ with
  their underlying proposition, and hence with their underlying types.
\end{definition}
\begin{remark}
  If $P$ and $Q$ are decidable, then so is $P\land Q$, and we use this
  fact in later developments.
\end{remark}
\begin{definition}
  The \emph{principle of excluded middle} is the claim that every
  proposition is decidable, that is:
  \[
    \PEM\coloneqq\dpt{P:\Prop}P+\neg P.
  \]
\end{definition}
\begin{lemma}\label{lem:locator:pem}
  Assuming $\PEM$, for every $x:\RD$, we can construct a locator for
  $x$.
\end{lemma}
\begin{proof}
  For given rationals $q<r$, use $\PEM$ to decide $q<x$.  If $q<x$
  holds, we can simply return the proof given by our application of
  $\PEM$.  If $\neg(q<x)$ holds, then we get $x\leq q< r$ so that we
  can return a proof of $x<r$.
\end{proof}

\begin{remark}
  Note that we use the word ``proof'' also to refer to type-theoretic
  constructions of types that are not propositions.  This section
  contains many such proofs that do not prove propositions in the
  sense of Definition~\ref{def:prop}.
\end{remark}

In Section~\ref{sec:analysis}, we will define when a
function $f:\RD\to\RD$ \emph{lifts to locators}, which can be seen as an
analogue to a computable function on the reals.  There, the contrast
with the theory of computable analysis becomes more pronounced, as
the notion of lifting to locators is neither stronger nor weaker than
continuity.

\medskip

The structure of a locator has been used previously by The Univalent
Foundations Program in a proof that assuming either countable choice
or excluded middle, the Cauchy reals and the Dedekind reals
coincide~\cite[Section~11.4]{hottbook}.

The reader may wonder why we only choose to modify one of the Dedekind
properties to become structure.  We show in
Theorem~\ref{thm:dedekind:struct:prop:rel} that given only a locator,
we can obtain the remaining structures, corresponding to boundedness,
roundedness and transitivity, automatically.

\subsection{Terminology for locators}
\label{sec:locator:terminology}

A locator $\ell$ for a real $x$ can be evaluated by picking $q,r:\Q$ and
$\nu:q<r$.  The value $\ell(q,r,\nu)$ has type $(q<x)+(x<r)$, and so
$\ell(q,r,\nu)$ can be either in the left summand or the right summand.  We
say that ``we locate $q<x$'' when the locator gives a value in the
left summand, and similarly we say ``we locate $x<r$'' when the
locator gives a value in the right summand.

We often do case analysis on $\ell(q,r,\nu) : (q<x)+(x<r)$ by
constructing a value $c:C(q<_xr)$ for some type family
$C:(q<x)+(x<r)\to\UU$.  To construct $c$ we use the elimination
principle of $+$, for which we need to specify two values
corresponding to the disjuncts $q<x$ and $x<r$, so the two values have
corresponding types $\dpt{\xi:q<x}C(\inl(\xi))$ and
$\dpt{\zeta:x<r}C(\inr(\zeta))$.  These two values correspond to the
two possible answers of the locator, and we will often indicate this
by using the above terminology: the expression ``we locate $q<x$''
corresponds to constructing a value of the former type, and the
expression ``we locate $x<r$'' corresponds to constructing a value of
the latter type.

For example, for every real $x$ with a locator $\ell$, we can output a
Boolean depending on whether $\ell$ locates $0<x$ or $x<1$.  Namely, if
we locate $0<x$ we output true, and if we locate $x<1$ we output
false.  We use this construction in the proof of
Lemma~\ref{lem:magic:struct:taboo}.

\subsection{Locators for rationals}
\label{sec:rationals}

\begin{lemma}\label{lem:rationals}
  Suppose $x:\RD$ is a rational, or more precisely, that
  $\ex{s:\Q}(x=i(s))$, with $i:\Q\hookrightarrow\RD$ the canonical
  embedding of the rationals into the Dedekind reals, then $x$ has a
  locator.
\end{lemma}
We give two constructions, to emphasize that locators are not unique.
We use trichotomy of the rationals, namely, for all $a,b:\Q$,
\[
  (a<b)+(a=b)+(a>b).
\]
\begin{proof}[First proof]
  Let $q<r$ be arbitrary, then we want to give $(q<s)+(s<r)$.  By
  trichotomy of the rationals applied to $q$ and $s$, we have
  \[
    (q<s)+(q=s)+(q>s)
  \]
  In the first case $q<s$, we can locate $q<s$.  In the second case
  $q=s$, we have $s=q<r$, so we locate $s<r$.  In the third case, we
  have $s<q<r$, so we locate $s<r$.
\end{proof}
\begin{proof}[Second proof]
  Let $q<r$ be arbitrary, then we want to give $(q<s)+(s<r)$.  By
  trichotomy of the rationals applied to $s$ and $r$, we have
  \[
    (s<r)+(s=r)+(s>r)
  \]
  In the first case $s<r$, we can locate $s<r$.  In the second case
  $s=r$, we have $q<r=s$, so we locate $q<s$.  In the third case, we
  have $q<r=s$, so we locate $q<s$.
\end{proof}
In the case that $q<s<r$, the first construction locates $s<r$,
whereas the second construction locates $q<s$.  In particular, given a
pair $q<r$ of rationals, the first proof locates $q<0$ if $q$ is
indeed negative, and $0<r$ otherwise.  The second proof locates $0<r$
if $r$ is indeed positive, and $q<0$ otherwise.  Note that these
locators disagree when $q<0<r$, illustrating that locators are not
unique.

\subsection{The logic of locators}
\label{sec:logic-locators}

Our aim is to combine properties of real numbers with the structure of
a locator to make discrete observations.

If one \emph{represents} reals by Cauchy sequences, one obtains lower
bounds immediately from the fact that any element in the sequence
approximates the real up to a known error.  As a working example, we
show, perhaps surprisingly, that we can get a lower bound for an real
$x$, that is an element of $\sm{q:\Q}q<x$, from the locator alone.

Recall that Dedekind reals are bounded from below, so that
$\ex{q:\Q}q<x$.  We will define a proposition $P$ which \emph{gives}
us a bound, in the sense that we can use the elimination rule for
propositional truncations to get a map
\[
  (\ex{q:\Q}q<x)\to P,
\]
and then we can extract a bound using a simple projection map
\[
  P\to(\sm{q:\Q}q<x).
\]

More concretely, we define a type of rationals which are bounds for
$x$ and which are \emph{minimal} in a certain sense.  The minimality
is \emph{not} intended to find tight bounds, but is intended to make
this collection of rationals into a proposition: in other words,
minimality ensures that the answer is unique, so that we may apply the
elimination rule for propositional truncations.

Our technique has two central elements: reasoning about the structure
of locators using propositions, and the construction of a
unique answer using bounded search (Section~\ref{sec:bounded-search}).

Given a locator $\ell:\SL(x)$, $q,r:\Q$ and $\nu:q<r$, we have the notation
\[
  q<_x^\ell r\coloneqq \ell(q,r,\nu): (q<x)+(x<r),
\]
leaving the proof of $q<r$ implicit.  We further often drop the choice
of locator, writing $q<_xr$ for $q<_x^\ell r$.

\begin{lemma}
  For types $A$ and $B$, we have
  \[
    A+B\simeq \sm{P:\DProp}(P\to A)\times(\neg P\to B).
  \]
\end{lemma}
\begin{proof}
  For a given element $x:A+B$, the proposition $P$ is defined to hold
  when $x$ an given by an element of $A$, and false otherwise, so that
  the two conditions on $P$ hold.  Vice versa, for a given proposition
  $P$ we simply decide $P$ to obtain the respective element of $A+B$.
  It has to be checked that these two constructions result in an
  equivalence.
\end{proof}
\begin{lemma}\label{lem:loc:dec}
  The type $\SL(x)$ of Definition~\ref{def:loc} is equivalent to the type
  \begin{align*}
     &\sm{\gL : \dpt{q, r : \Q} q < r \to \DProp} \\
    &\quad\phantom{{}\times{}}(\dpt{q, r : \Q} \dpt{\nu : q < r} \gL( q, r, \nu) \to q < x) \\
    &\quad\times(\dpt{q, r : \Q} \dpt{\nu : q < r} \neg \gL (q, r, \nu) \to x
      < r) .
  \end{align*}
\end{lemma}
\begin{proof}
  The previous lemma yields the equivalence
  \begin{align*}
    \SL(x)\simeq{} & \dpt{q, r : \Q} q < r \to\\
                &\quad\sm{P:\DProp}(P\to q<x)\times(\neg P \to x<r),
  \end{align*}
  and then we can apply Theorems~2.15.5 and~2.15.7 in The Univalent Foundations
  Program~\cite{hottbook} to distribute the $\Pi$-types over $\Sigma$ and
  $\times$.
\end{proof}
\begin{remark}
  We emphasize that, confusingly, $\gL( q, r, \nu)$ is defined
  type\hyp{}theoretically as $\operatorname{isLeft}(q<_x^\ell r)$.
\end{remark}

\begin{definition}\label{def:gives:lower:upper}
  For a real $x$ with a locator $\ell$ and rationals $q<r$, we write
  \[
    \gL(q<_x^\ell r) \qquad\text{or}\qquad
    \gL(q<_xr)
  \]
  for the decidable proposition $\gL(q,r,\nu)$ obtained from
  Lemma~\ref{lem:loc:dec}.
  We write
  \[
    \gU(q<_x^\ell r) \qquad\text{or}\qquad
    \gU(q<_xr)
  \]
  to be the negation of $\gL(q<_xr)$: so it is the proposition which
  is true if we locate $x<r$.
\end{definition}

\begin{remark}
  In general, if we have $q'<q<r$, then $\gL({q<_xr})$ does \emph{not}
  imply $\gL(q'<_xr)$.
\end{remark}

\begin{lemma}\label{lem:g:dec}
  For any real $x$ with a locator $\ell$ and rationals $q<r$,
  \begin{align*}
    \neg(q<x)&\Implies\gU(q<_x^\ell r),\qquad\text{and}\\
    \neg(x<r)&\Implies\gL(q<_x^\ell r).
  \end{align*}
\end{lemma}
\begin{proof}
  From the defining properties of $\gL$ in Lemma~\ref{lem:loc:dec}, we
  know
  \begin{align*}
    \gL( q<_x^\ell r) &\Implies (q < x),\qquad\text{and}\\
    \neg\gL (q<_x^\ell r)&\Implies(x < r) .
  \end{align*}
  The contrapositives of these are, respectively:
  \begin{align*}
    \neg(q<x)&\Implies\neg\gL( q<_x^\ell r),\qquad\text{and}\\
    \neg(x<r)&\Implies\neg\neg\gL (q<_x^\ell r).
  \end{align*}
  Using the fact that $\neg\neg A\Implies A$ when $A$ is decidable,
  this is the required result.
\end{proof}
\begin{example}
  Let $x$ be a real equipped with a locator.  We can
  type-theoretically express that the locator must give certain
  answers.  For example, if we have $q<r<x$, shown visually as
  \begin{center}
    \begin{tikzpicture}
      \draw[->,semithick] (0,0) -- (-3,0); \draw[|->,semithick] (0,0)
      -- (3,0); \draw[|->,semithick] (-2,0) -- (3,0);
      \draw[|->,semithick] (-1,0) -- (3,0); \draw (0, -0.3) node
      {$x$}; \draw (-2,-0.3) node {$q$}; \draw (-1.5,-0.3) node {$<$};
      \draw (-1,-0.3) node {$r$}; \draw (-4, 0) node {$\RD$};
    \end{tikzpicture}
  \end{center}
  we must locate $q<x$, because $\neg(x<r)$.  In other words, we
  obtain truth of the proposition $\gL(q<_xr)$: the property
  $\neg(x<r)$ yielded a property of the structure $q<_xr$.
\end{example}

Continuing our working example of computing a lower bound, for any
$q:\Q$ we have the claim
\[
  P(q)\coloneqq\gL(q-1<_xq)
\]
that we locate $q-1<x$.  This claim is a decidable proposition.  And
from the existence $\ex{q:\Q}q<x$ of a lower bound for $x$, we can
deduce that $\ex{q:\Q}P(q)$, because if $q<x$ then $\neg(x<q)$ and
hence the above lemma applies.  If we manage to \emph{find} a $q:\Q$
for which $P(q)$ holds, then we have certainly found a lower bound of
$x$, namely $q-1$.

\subsection{Bounded search}
\label{sec:bounded-search}

Even though the elimination rule for propositional truncation in
Definition~\ref{def:truncation} only constructs maps into
propositions, we can use elements of propositional truncations to
obtain witnesses of non-truncated types --- in other words, we can
sometimes obtain structure from property.

\begin{theorem}[Escard\'o~\cite{escardo:nat:dec},~\cite{escardo:xu:inconsistency},~{\cite[Exercise~3.19]{hottbook}}]\label{thm:nat-dec}
  Let $P:\N\to\DProp$.  If $\ex{n:\N}P(n)$ then we can construct an
  element of $\sm{n:\N}P(n)$.
\end{theorem}

\begin{remark}
  In general, we don't have $\trunc{X}\to X$ for all types $X$, as
  this would imply excluded
  middle~\cite{DBLP:journals/corr/KrausECA16}.  But for some types
  $X$, we do have $\trunc{X}\to X$, namely when $X$ has a constant
  endomap~\cite{DBLP:journals/corr/KrausECA16}.
\end{remark}

Even without univalence, Theorem~\ref{thm:nat-dec} also works for any
type equivalent to $\N$.
\begin{corollary}\label{cor:enum:dec}
  Let $A$ be a type and $e:\N\eqv A$ be an equivalence, that is, a
  function $\N\to A$ with a left inverse and right inverse.  Let
  $P:A\to\DProp$.  If $\ex{a:A}P(a)$ then we can construct an element
  of $\sm{a:A}P(n)$.
\end{corollary}
\begin{proof}
  Use Theorem~\ref{thm:nat-dec} with $P'(n)\coloneqq P(e(n))$.  In
  order to show $\ex{n:\N}P'(n)$, it suffices to show
  $\left(\sm{a:A}P(a)\right)\to\left(\sm{n:\N}P'(n)\right)$, so let
  $a:A$ and $p:P(a)$.  Then since $a=e(e^{-1}(a))$ we get
  $P(e(e^{-1}(a)))$ by transport.

  Hence from Theorem~\ref{thm:nat-dec} we obtain some
  $(n,p'):\sm{n:\N}P'(e(n))$, so we can output $(e(n),q)$.
\end{proof}

\subsection{Computing bounds}
\label{sec:computing-bounds}

We are now ready to finish our running example of
computing a lower bound for $x$.

\begin{lemma}\label{lem:bounders}
  Given a real $x:\RD$ equipped with a locator, we get bounds
  for $x$, that is, we can find $q,r:\Q$ with $q<x<r$.
\end{lemma}
\begin{proof}
  We pick any enumeration of $\Q$, that is, an equivalence
  $\N\eqv\Q$.  Set
  \[
    P(q)\coloneqq\gL(q-1<_xq).
  \]
  From Section~\ref{sec:logic-locators} we know that $\ex{q:\Q}P(q)$,
  and so we can apply Corollary~\ref{cor:enum:dec}.  We obtain
  $\sm{q:\Q}P(q)$, and in particular $\sm{q:\Q}q-1<x$.

  Upper bounds are constructed by a symmetric argument, using
  \[
    P(r)\coloneqq\gU(r<_xr+1).\qedhere
  \]
\end{proof}

We emphasize that even though we cannot decide $q<x$ in general, we
\emph{can} decide what the locator tells us, and this is what is
exploited in our development.  Given a real $x$ with a locator, the
above construction of a lower bound searches for a rational $q$ for
which we locate $q-1<x$.  We emphasize once more that the
rational thus found is minimal in the sense that it appears first in
the chosen enumeration of $\Q$, and \emph{not} a tight bound.

\begin{remark}\label{rem:bounders}
  The proof of Theorem~\ref{thm:nat-dec} works by an exhaustive, but
  bounded, search.  So our construction for Lemma~\ref{lem:bounders}
  similarly exhaustively searches for an appropriate rational $q$.
  The efficiency of the algorithm thus obtained can be improved:
  \begin{enumerate}
  \item We do not need to test every rational number: it suffices to
    test, for example, bounds of the form $\pm 2^{k+1}$ for $k:\N$, as
    there always exists a bound of that form.  Formally, such a
    construction is set up by enumerating a subset of the integers
    instead of enumerating all rationals, and showing the existence of
    a bound of the chosen form, followed by application of
    Corollary~\ref{cor:enum:dec}.
  \item More practically, Lemma~\ref{lem:bounders} shows that we may
    as well additionally equip bounds to reals that already have
    locators.  Then, any later constructions that use rational bounds
    can simply use these equipped rational bounds.  This is
    essentially the approach of interval arithmetic with open
    nondegenerate intervals.  We can also see this equipping of bounds
    as a form of memoization, which we can apply more generally.
  \end{enumerate}
\end{remark}

\begin{lemma}\label{lem:precise}
  For a real $x$ equipped with a locator and any positive rational
  $\varepsilon$ we can find $u,v:\Q$ with $u<x<v$ and
  $v-u<\varepsilon$.
\end{lemma}
\begin{proof}
  The construction of bounds in Lemma~\ref{lem:bounders} yields
  $q,r:\Q$ with $q<x<r$.  We can compute $n:\N$ such that
  $r<q+\frac{n\varepsilon}3$.  Consider the equidistant subdivision
  \[
    q-\frac\varepsilon3,\,\,q,\,\,q+\frac\varepsilon3,\,\,q+\frac{2\varepsilon}3,\,\,\ldots,\,\,q+\frac{n\varepsilon}3,\,\,q+\frac{(n+1)\varepsilon}3.
  \]
  By Lemma~\ref{lem:g:dec}, necessarily $\gL(q-\frac\varepsilon3<_xq)$
  because $q<x$.  Similarly, we have
  $\gU{(q+\frac{n\varepsilon}3<_xq+\frac{(n+1)\varepsilon}3)}$ because
  $x<q+\frac{n\varepsilon}3$.

  For some $i$, which we can find by a finite search using a
  one-dimensional version of Sperner's lemma, we have
  \[
    \gL\left(q+\frac{i\varepsilon}3<_xq+\frac{(i+1)\varepsilon}3\right)
    \land
    \gU\left(q+\frac{(i+1)\varepsilon}3<_xq+\frac{(i+2)\varepsilon}3\right).
  \]
  For this $i$, we can output $u=q+\frac{i\varepsilon}3$ and
  $v=q+\frac{(i+2)\varepsilon}3$.
\end{proof}
\begin{remark}
  The above result allows us to compute arbitrarily precise bounds for
  a real number $x$ with a locator.  But, as in
  Remark~\ref{rem:bounders}, the above theorem shows that we may as
  well \emph{equip} an appropriate algorithm for computing arbitrarily
  precise lower and upper bounds to real numbers.  This may be a
  better idea when efficiency of the computation matters.
\end{remark}

\subsection{Locators for algebraic operations}
\label{sec:algebraic-operations}

If $x$ and $y$ are reals that we can compute with in an appropriate
sense, then we expect to be able to do so with $-x$, $x+y$,
$x\cdot y$, $x^{-1}$ (assuming $x\apart0$), $\min(x,y)$ and
$\max(x,y)$ as well.  In our case, that means that if $x$ and $y$ come
equipped with locators, then so should the previously listed values.

If one works with intensional real numbers, such as when they are
given as Cauchy sequences, then the algebraic operations are specified
directly on the representations.  This means that the computational
data is automatically present.  Since in our case the algebraic
operations are specified extensionally, they do not give any discrete
data, and so the construction of locators has to be done explicitly in
order to compute.

The algebraic operations can be defined for Dedekind cuts as in The
Univalent Foundations Program~\cite[Section~11.2.1]{hottbook}.  Recall
from Section~\ref{sec:dedekind-reals} that for a Dedekind cut
$x=(L,U)$ we write $q<x$ for the claim that $q:\Q$ is in the left cut
$L$.  In fact, now that we have identified $q:\Q$ with its canonical
embedding $i(q):\RD$ in the reals, we can simply understand $q<x$ as
$i(q)<_\RD x$, which coincides with the notation for Dedekind cuts.
In summary, we have the following relations for $x,y,z,w:\RD$ with
$w<0<z$ and $q,r:\Q$:
\begin{align*}
  q<-x&\Iff x<-q\\
  -x<r&\Iff -r<x\\
  q<x+y&\Iff \ex{s:\Q}s<x\land(q-s)<y\\
  x+y<r&\Iff \ex{t:\Q}x<t\land y<(r-t)\\\displaybreak[2]
  q<xy &\Iff \ex{a,b,c,d:\Q}q<\min(ac,ad,bc,bd)\\\displaybreak[0]
      &\qquad\land a<x<b\land c<y<d\\
  xy<r &\Iff \ex{a,b,c,d:\Q}\max(ac,ad,bc,bd)<r\\\displaybreak[0]
      &\qquad\land a<x<b\land c<y<d\\\displaybreak[2]
  q<z^{-1} &\Iff qz<1\\
  z^{-1}<r &\Iff 1<rz\\
  q<w^{-1} &\Iff 1<qw\\
  w^{-1}<r &\Iff rw<1\\\displaybreak[2]
  q<\min(x,y)&\Iff q<x\land q<y\\
  \min(x,y)<r&\Iff x<r\lor y<r\\
  q<\max(x,y)&\Iff q<x\lor q<y\\
  \max(x,y)<r&\Iff x<r\land y<r
\end{align*}

The Dedekind reals satisfy the Archimedean property, which can be
succinctly stated as the claim that for all $x,y:\RD$,
\[
  x<y\Implies\ex{q:\Q}x<q<y.
\]

We will use the following variation of the Archimedean property.
We write $\Q_+$ for the positive rationals.
\begin{lemma}\label{lem:arch:twice}
  For real numbers $x<y$, there \emph{exist} $q:\Q$ and
  $\varepsilon:\Q_+$ with $x<q-\varepsilon<q+\varepsilon<y$.
\end{lemma}
\begin{proof}
  By a first application of the Archimedean property, we know
  $\ex{s:\Q}x<s<y$.  Since we are showing a proposition, we may assume
  to have such an $s:\Q$.  Now for $s<y$, by the Archimedean property,
  we know $\ex{t:\Q}s<t<y$, and again we may assume to have such a
  $t$.  Now set $q\coloneqq\frac{s+t}{2}$ and
  $\varepsilon\coloneqq\frac{t-s}{2}$.
\end{proof}
In particular, the above variation can be used to strengthen the
$\exists$ of the Archimedean property into $\Sigma$ when the reals
involved come equipped with locators.  Its corollary,
Corollary~\ref{cor:cotrans:rat}, is used to compute locators for
multiplicative inverses.
\begin{lemma}\label{lem:archimedean:struct}
  For reals $x$ and $y$ equipped with locators we have the Archimedean
  \emph{structure}
  \[
    x<y\to\sm{q:\Q}x<q<y.
  \]
\end{lemma}
\begin{proof}
  Let $x$ and $y$ be reals equipped with locators.  By
  Lemma~\ref{lem:arch:twice}, there exist $q:\Q$ and
  $\varepsilon:\Q_+$ with $x<q-\varepsilon<q+\varepsilon<y$.  The
  following proposition is decidable for any $(q',\varepsilon')$ and
  we have $\ex{(q,\varepsilon):\Q\times\Q_+}P(q,\varepsilon)$:
  \[
    P(q',\varepsilon') \coloneqq
    \gU(q'-\varepsilon'<_xq') \land \gL(q'<_yq'+\varepsilon').
  \]
  Using Corollary~\ref{cor:enum:dec} we can find $(q',\varepsilon')$
  with $P(q',\varepsilon')$ and hence $x<q'<y$.
\end{proof}

\begin{corollary}\label{cor:cotrans:rat}
  For reals $x$ and $y$ equipped with locators, and $s:\Q$ a rational,
  if $x<y$ then we have a choice of $x<s$ or $s<y$, that is:
  \[
    \dpt{s:\Q}x<y\to (x<s) + (s<y).
  \]
\end{corollary}
\begin{proof}
  By Lemma~\ref{lem:archimedean:struct} we can find $q:\Q$ with
  $x<q<y$.  Apply trichotomy of the rationals: if $q<s$ or $q=s$ then
  we locate $x<s$, and if $s<q$ then we locate $s<y$.
\end{proof}
\begin{remark}\label{rem:cotransitivity}
  Instead of the rational $s:\Q$ we can have any real $z$ equipped
  with a locator in the above corollary, so that we obtain a form of
  strong cotransitivity of the strict ordering relation on the real
  numbers, but we will not be using this.

  Having developed such a strong cotransitivity, we could characterize
  the algebraic operations on the Dedekind reals using the Archimedean
  \emph{structure} of Lemma~\ref{lem:archimedean:struct} rather than
  using the Archimedean property.  This would then yield a structural
  characterization of the algebraic operations for $x,y:\RD$ equipped
  with locators, along the lines of:
\begin{alignat*}{2}
  q&<x+y &\quad\Iff\quad& \sm{s,t:\Q} (q=s+t)\wedge s<x\wedge t<y
  \\
  %
  q&<x\cdot y &\quad\Iff\quad& \sm{a,b,c,d:\Q} q<\min(a\cdot c,a\cdot
                            d,b\cdot c,b\cdot d)
  \\
  &&&\qquad\land a<x<b \wedge c<y<d
  \\
  q&<\max(x,y)&\quad\Iff\quad& q<x + q<y
  \\
     &&\vdots\quad&
\end{alignat*}
\end{remark}
\begin{theorem}\label{thm:ops}
  If reals $x,y:\RD$ are equipped with locators, then we can also equip
  $-x$, $x+y$, $x\cdot y$, $x^{-1}$ (assuming $x\apart0$), $\min(x,y)$
  and $\max(x,y)$ with locators.
\end{theorem}
\begin{remark}
  As we define absolute values by $|x|=\max(x,-x)$, as is common in
  constructive analysis, if $x$ has a locator, then so does $|x|$, and
  we use this fact in the proof of the above theorem.
\end{remark}
\begin{proof}[Proof of Theorem~\ref{thm:ops}]
  Throughout this proof, we assume $x$ and $y$ to be reals equipped
  with locators, and $q<r$ to be rationals.

  \medskip

  We construct a locator for $-x$.  We can give $(q<-x)+(-x<r)$ by
  considering $-r<_x-q$.

  \medskip

  We construct a locator for $x+y$.  We need to show
  $(q<x+y)+ (x+y<r)$.  Note that $q<x+y$ iff there exists $s:\Q$ with
  $q-s<x$ and $s<y$.  Similarly, $x+y<r$ iff there exists $t:\Q$ with
  $x<r-t$ and $y<t$.

  Set $\varepsilon\coloneqq (r-q)/2$, such that
  $q+\varepsilon=r-\varepsilon$.  By Lemma~\ref{lem:precise} we can
  find $u,v:\Q$ such that $u<x<v$ and $v-u<\varepsilon$, so in
  particular $x<u+\varepsilon$.  Set $s\coloneqq q-u$, so that
  $q-s<x$.  Now consider $s<_ys+\varepsilon$.  If we locate $s<y$,
  we locate $q<x+y$.  If we locate $y<s+\varepsilon$, we have
  $x<q-s+\varepsilon=r-s-\varepsilon$, that is, we can set
  $t\coloneqq s+\varepsilon$ to locate $x+y<r$.

  \medskip

  We construct a locator for $\min(x,y)$.  We consider both $q<_xr$
  and $q<_yr$.  If we locate $x<r$ or $y<r$, we can locate
  $\min(x,y)<r$.  Otherwise, we have located both $q<x$ and $q<y$, so
  we can locate $q<\min(x,y)$.

  \medskip

  The locator for $\max(x,y)$ is symmetric to the case of $\min(x,y)$.

  \medskip

  We construct a locator for $xy$.  We need to show $(q<xy)+ (xy<r)$.
  Note that $q<xy$ means:
  \[
    \ex{a,b,c,d:\Q}(a<x<b) \land (c<y<d) \land
    (q<\min\{ac,ad,bc,bd\}).
  \]
  Similarly, $xy<r$ means:
  \[
    \ex{a,b,c,d:\Q}(a<x<b) \land (c<y<d) \land
    (\max\{ac,ad,bc,bd\}<r).
  \]

  Using Lemma~\ref{lem:precise} we can find $z,w:\Q$ with $|x|+1<z$ and
  $|y|+1<w$, since we have already constructed locators for $\max$,
  $+$, $-$ and all rationals.

  Set $\varepsilon\coloneqq r-q$,
  $\delta\coloneqq\min\{1,\frac{\varepsilon}{2z}\}$ and
  $\eta\coloneqq\min\{1,\frac{\varepsilon}{2w}\}$ .  Find $a<x<b$ and
  $c<y<d$ such that $b-a<\eta$ and $d-c<\delta$.  Note that
  $|a|<|x|+\eta\leq|x|+1<z$ and similarly $|b|<z$, $|c|<w$ and
  $|d|<w$.  Then the distance between any two elements of
  $\{ac,ad,bc,bd\}$ is less than $\varepsilon$.  For instance,
  $|ac-bd|<\varepsilon$ because $|ac-bd|\leq|ac-ad|+|ad-bd|$, and
  $|ac-ad|=|a||c-d|<|a|\delta<\frac\varepsilon2$ and similarly
  $|ad-bd|<\frac\varepsilon2$.  Hence
  $\max\{ac,ad,bc,bd\}-\min\{ac,ad,bc,bd\}<\varepsilon$.  Thus, by
  dichotomy of the rationals, one of $q<\min\{ac,ad,bc,bd\}$ and
  $\max\{ac,ad,bc,bd\}<r$ must be true, yielding a corresponding
  choice of $(q<xy)+ (xy<r)$.

  \medskip

  We construct a locator for $x^{-1}$.  Consider the case that $x>0$.
  Given $q<r$, we need $(q<x^{-1})+(x^{-1}<r)$, or equivalently
  $(qx<1)+(1<rx)$.  By the previous case, $qx$ and $rx$ have locators,
  so we can apply Corollary~\ref{cor:cotrans:rat}.  The case $x<0$ is
  similar.
\end{proof}
This proof works whether we use a definition of algebraic operations
as in The Univalent Foundations Program~\cite{hottbook}, or whether we
work with the archimedean field axioms, because from the archimedean
field axioms we deduce the same properties as the definitions.
\begin{remark}
  Locators for reciprocals can also be constructed by more elementary
  methods, as follows.  For $x>0$, we use dichotomy of the rationals
  for $0$ and $q$.  If $q\leq0$ we may locate $q<x$, and otherwise we
  have $0<1/r<1/q$, so that by considering $1/r<_x1/q$ we may either
  locate $x<r$ or $q<x$.  There is a similar construction for $x<0$.
\end{remark}

By using the techniques of Sections~\ref{sec:logic-locators}
and~\ref{sec:bounded-search}, we have computed locators for algebraic
operations applied to reals equipped with locators.

\subsection{Locators for limits}
\label{sec:locators:limits}

In a spirit similar to the previous section, if we have a Cauchy
sequence of reals, each of which equipped with a locator, then we can
compute a locator for the limit of the sequence.

\begin{lemma}\label{lem:lim:locator}
  Suppose $x:\N\to\RD$ has modulus of Cauchy convergence $M$, and
  suppose that every value in the sequence $x:\N\to\RD$ comes equipped
  with a locator, that is, suppose we have an element of
  \( \dpt{n:\N}\SL\left(x_n\right).  \) Then we have a locator for
  $\lim_{n\to\infty}x_n$.
\end{lemma}
\begin{proof}
  Let $q<r$ be arbitrary rationals.  We need
  $(q<\lim_{n\to\infty}x_n) + (\lim_{n\to\infty}x_n<r)$.  Set
  \( \varepsilon\coloneqq \frac{r-q}{3} \) so that
  $q+\varepsilon<r-\varepsilon$.  Since $M$ is a modulus of Cauchy
  convergence, we have
  $\abs{x_{M(\varepsilon/2)}-\lim_{n\to\infty}x_n}<\varepsilon$, that
  is
  \[
    x_{M(\varepsilon/2)}-\varepsilon<\lim_{n\to\infty}x_n<x_{M(\varepsilon/2)}+\varepsilon.
  \]
  We consider the locator equipped to $x_{M(\varepsilon/2)}$ and do
  case analysis on
  $q+\varepsilon<_{x_{M(\varepsilon/2)}}r-\varepsilon$.  If we locate
  $q+\varepsilon<x_{M(\varepsilon/2)}$ then we can locate
  $q<\lim_{n\to\infty}x_n$.  If we locate
  $x_{M(\varepsilon/2)}<r-\varepsilon$ then we can locate
  $\lim_{n\to\infty}x_n<r$.
\end{proof}
\begin{remark}
  We emphasize that Lemma~\ref{lem:lim:locator} requires the sequence
  to be \emph{equipped} with a modulus of Cauchy convergence, whereas
  existence suffices for the computation of the limit
  $\lim_{n\to\infty}x_n$ itself, namely the element of $\RD$.
\end{remark}
\begin{example}[Locators for exponentials]\label{ex:exp:limit}
  Given a locator for $x$, we can use Lemma~\ref{lem:bounders} to
  obtain a modulus of Cauchy convergence of
  $\exp(x)=\sum_{k=0}^\infty\frac{x^k}{k!}$.  Hence $\exp(x)$ has a
  locator.
\end{example}
\begin{example}\label{ex:constants}
  Many constants such as $\pi$ and $e$ have locators, which can be
  found by examining their construction as limits of sequences.
\end{example}

We can now construct locators for limits of sequences whose elements
have locators, and so using Lemma~\ref{lem:rationals}, in particular,
limits for sequences of rationals.  As we will make precise in
Theorem~\ref{thm:locator:characterization}, this covers all the cases.

\subsection{Calculating digits}
\label{sec:sdr}

\begin{example}
  We would like to print digits for numbers equipped with locators,
  such as $\pi$.  Such a digit expansion gives rise to rational bounds
  of the number in question: if a digit expansion of $\pi$ starts with
  $3.1\ldots$, then we have the bounds $3.0<\pi<3.3$.
\end{example}

We now wish to generate the entire sequence of digits of a real number
$x$ equipped with a locator.  As in computable analysis and other
settings where one works intensionally, with reals given as Cauchy
sequences or streams of digits, we wish to extract digit
representations from a real equipped with a locator.

In fact, various authors including Brouwer~\cite{brouwer:zahl} and
Turing~\cite{turing:correction} encountered problems with computing
decimal expansions of real numbers in their work.  As is common in
constructive analysis, we instead consider \emph{signed}-digit
representations.  Wiedmer shows how to calculate directly on the
signed-digit representations in terms of computability
theory~\cite{wiedmer1977exaktes}.

\begin{definition}
  A signed-digit representation for $x:\RD$ is given by $k:\Z$ and a
  sequence $a$ of signed digits
  $a_i\in\set{\bar{9},\bar{8},\ldots,\bar{1},0,1,\ldots,9}$, with
  $\bar{a}\coloneqq-a$, such that
  \[
    x=k+\sum_{i=0}^\infty a_i\cdot 10^{-i-1}.
  \]
\end{definition}
\begin{example}
  The number $\pi$ may be given by a signed-decimal expansion as
  $3.1415\ldots$, or as $4.\bar{8}\bar{6}15\ldots$, or as
  $3.2\bar{5}\bar{8}\bar{5}\ldots$.
\end{example}

\begin{lemma}\label{lem:pre:digit}
  For any $x$ equipped with a locator, we
  can find $k:\Z$ such that $x\in(k-1,k+1)$.
\end{lemma}
\begin{proof}
  Use Lemma~\ref{lem:precise} with $\varepsilon=1$ to obtain
  rationals $u<v$ with $u<x<v$ and $v<1+u$.  Set $k=\floor{u}+1$. Then:
  \[
    k-1=\floor{u}\leq u<x<v<u+1<k+1.\qedhere
  \]
\end{proof}

\begin{theorem}\label{thm:sdr}
  For a real number $x$, locators and signed-digit representations are
  interdefinable.
\end{theorem}
\begin{proof}
  If a real number has a signed-digit representation, then it is the
  limit of a sequence of rational numbers, and so by
  Lemma~\ref{lem:lim:locator} it has a locator.

  Conversely, assume a real $x$ has a locator.  By
  Lemma~\ref{lem:pre:digit} we get $k:\Z$ with $x\in(k-1,k+1)$.
  Consider the equidistant subdivision
  \[
    k-1<k-\frac{9}{10}<\cdots<k-\frac{1}{10}<k<k+\frac{1}{10}<\cdots<k+1.
  \]
  By applying the locator several times, we can find a signed digit
  $a_0$ such that
  \[
    k+\frac{a_0-1}{10}<x<k+\frac{a_0+1}{10}.
  \]
  We find subsequent digits in a similar way.
\end{proof}

Note that since $\RD$ is Cauchy complete, there is a canonical
inclusion $\RC\to\RD$ from the Cauchy reals into $\RD$.

\newcommand{\isCauchyReal}{\operatorname{isCauchyReal}}
\begin{definition}
  We write $\isCauchyReal(x)$ for the claim that a given real $x:\RD$
  is in the image of the canonical inclusion of the Cauchy reals into
  $\RD$.  Equivalently, $\isCauchyReal(x)$ holds when there is a
  rational Cauchy sequence with limit $x$.
\end{definition}

We emphasize that $\trunc{\SL(x)}$ is \emph{not} equivalent to the
locatedness property of Definition~\ref{def:dedekind}.
\begin{theorem}\label{thm:locator:characterization}
  The following are equivalent for $x:\RD$:
  \begin{enumerate}
  \item $\trunc{\SL(x)}$, that is, there exists a locator for $x$.
  \item There exists a signed-digit representation of $x$.
  \item There exists a Cauchy sequence of rationals that $x$ is the
    limit of.
  \item $\isCauchyReal(x)$.
  \end{enumerate}
\end{theorem}
\begin{proof}
  Items~1 and~2 are equivalent by Theorem~\ref{thm:sdr}.  Item~2
  implies item~3 since a signed-digit representation gives rise to a
  sequence with a modulus of Cauchy convergence.  Item~3 implies
  item~1 because a sequence of rational numbers with modulus of Cauchy
  convergence has a locator by Lemma~\ref{lem:lim:locator}.
  Equivalence of items~3 and~4 is a standard result.
\end{proof}
\begin{remark}
  The notion of locator can be truncated into a proposition in three
  ways:
  \begin{align}
    \trunc{\dpt{q,r:\Q}q<r\to(q<x)+(x<r)} \label{eq:trunc:outer}\\
    \dpt{q,r:\Q}\trunc{q<r\to(q<x)+(x<r)} \label{eq:trunc:mid} \\
    \dpt{q,r:\Q}q<r\to\trunc{(q<x)+(x<r)} \label{eq:trunc:inner}
  \end{align}
  Now \eqref{eq:trunc:outer} is $\trunc{\SL(x)}$, and
  \eqref{eq:trunc:inner} is the locatedness property of
  Definition~\ref{def:dedekind}, which holds for all $x:\RD$ as
  mentioned in Section~\ref{sec:locator:definition}.  In summary, we
  have
  \[
    \text{\eqref{eq:trunc:outer}}\implies\text{\eqref{eq:trunc:mid}}\iff\text{\eqref{eq:trunc:inner}}
  \]
  where the implications to the right can be shown using the induction
  rule for propositional truncations, and the implication to the left
  follows from the fact that $q<r$ is a decidable proposition for
  $q,r:\Q$.
\end{remark}

In other words, we cannot expect to be able to equip every real with a
locator, as this would certainly imply that the Cauchy reals and the
Dedekind reals coincide, which is not true in
general~\cite{lubarsky:cauchy}.

\begin{corollary}
  The following are equivalent:
  \begin{enumerate}
  \item For every Dedekind real there exists a signed-digit
    representation of it.
  \item The Cauchy reals and the Dedekind reals coincide.
  \end{enumerate}
\end{corollary}

The types $\RC$ and $\RD$ do not coincide in general, but they do assuming
excluded middle or countable choice.  We are not aware of a classical
principle that is equivalent with the coincidence of $\RC$ and $\RD$.

\subsection{Dedekind cuts structure}
\label{sec:cuts-as-structure}

Let $x=(L,U)$ be a pair of predicates on the rationals, i.e.\
$L,U:\pow\Q$.  In Definition~\ref{def:dedekind} we specified the
necessary \emph{properties} for $x$ to be a Dedekind cut.  More
explicitly, we have $\isCut:\pow\Q\times\pow\Q\to\Prop$
defined by:
\begin{align*}
  \isCut(x)\coloneqq{}&\boundedLower(x)\land\boundedUpper(x)\\
  \land{}&\closedLower(x)\land\closedUpper(x)\\
  \land{}&\openLower(x)\land\openUpper(x)\\
  \land{}&\transitive(x)\land\located(x)
\end{align*}
where
\begin{align*}
  \boundedLower(x)\coloneqq{}&\ex{q : \Q} q<x,\\
  \boundedUpper(x)\coloneqq{}&\ex{r : \Q} x<r,\displaybreak[1]\\
  \closedLower(x)\coloneqq{}&\fa{q,q':\Q}(q < q') \land (q'<x)\Implies q<x,\\
  \closedUpper(x)\coloneqq{}&\fa{r,r':\Q}(r' < r) \land (x<r')\Implies x<r,\displaybreak[2]\\
  \openLower(x)\coloneqq{}&\fa{q:\Q}q<x\Implies \ex{q' : \Q} (q < q') \land (q'<x),\\
  \openUpper(x)\coloneqq{}&\fa{r:\Q}x<r \Implies\ex{r' : \Q} (r' < r) \land (x<r'),\displaybreak[1]\\
  \transitive(x)\coloneqq{}&\fa{q,r:\Q}(q<x)\land (x<r)\Implies (q<r),\\
  \located(x)\coloneqq{}&\fa{q,r:\Q}(q < r) \Implies (q<x) \lor (x<r).\\
\end{align*}

We may also consider when $x$ has these data as \emph{structure}, that
is, when it is equipped with the structure
$\cutStruct:\pow\Q\times\pow\Q\to\UU$ defined by:
\begin{align*}
  \cutStruct(x)\coloneqq{}&\bounderLower(x)\times\bounderUpper(x)\\
  \times&\closerLower(x)\times\closerUpper(x)\\
  \times&\openerLower(x)\times\openerUpper(x)\\
  \times&\transitor(x)\times\locator(x)
\end{align*}
where
\begin{align*}
  \bounderLower(x)\coloneqq{}&\sm{q : \Q} q<x,\\
  \bounderUpper(x)\coloneqq{}&\sm{r : \Q} x<r,\displaybreak[1]\\
  \closerLower(x)\coloneqq{}&\dpt{q,q':\Q}(q < q') \times (q'<x)\to q<x,\\
  \closerUpper(x)\coloneqq{}&\dpt{r,r':\Q}(r' < r) \times (x<r')\to x<r,\displaybreak[2]\\
  \openerLower(x)\coloneqq{}&\dpt{q:\Q}q<x\to \sm{q' : \Q} (q < q') \times (q'<x),\\
  \openerUpper(x)\coloneqq{}&\dpt{r:\Q}x<r \to\sm{r' : \Q} (r' < r) \times (x<r'),\displaybreak[1]\\
  \transitor(x)\coloneqq{}&\dpt{q,r:\Q}(q<x)\times (x<r)\to (q<r),\\
  \locator(x)\coloneqq{}&\dpt{q,r:\Q}(q < r) \to (q<x) + (x<r)=\SL(x).\\
\end{align*}

In this section we investigate when $x=(L,U)$ has the property
$\isCut(x)$, and when it has the data $\cutStruct(x)$.  First, note
that we cannot expect all Dedekind cuts to come equipped with that
data.

\begin{lemma}\label{lem:magic:struct:taboo}
  Suppose given a choice $\dpt{x:\RD}\SL(x)$ of locator for each
  $x:\RD$.  Then we can define a strongly non-constant function
  $f:\RD\to\bool$ in the sense that there exist reals $x,y:\RD$ with
  $f(x)\neq f(y)$.
\end{lemma}
\begin{proof}
  Given a locator for $x:\RD$, we can output true or false depending on
  whether the locator return the left or the right summand for $0<1$,
  as follows.
  \[
    f(x)=
    \begin{cases}
      \text{true}&\text{if $\gL(0<_x1)$}\\
      \text{false}&\text{if $\gU(0<_x1)$}.
    \end{cases}
  \]
  The map thus constructed must give a different answer for the real
  numbers $0$ and $1$.
\end{proof}
Since any strongly non-constant map from the reals to the Booleans
gives rise to a discontinuous map on the reals, we have violated the
continuity principle that every map on the reals is continuous.
Following Ishihara~\cite{ishihara:sequentially}, we can derive $\WLPO$
from it.
\begin{definition}
  The \emph{weak limited principle of omniscience} is the following
  consequence of $\PEM$: for every decidable predicate $P:\N\to\DProp$
  on the naturals, we can decide $\neg\ex{n:\N}P(n)$:
  \[
    \WLPO\coloneqq\dpt{P:\N\to\DProp}\neg(\ex{n:\N}P(n))+\neg\neg(\ex{n:\N}P(n)).
  \]
  Note that this is a proposition in the sense of Definition~\ref{def:prop}
  because if $Q$ is one then $Q+\neg Q$ is one.
\end{definition}
\begin{lemma}\label{lem:disc:wlpo}
  If there exists a strongly non-constant function $\RD\to\bool$, then
  $\WLPO$ holds.
\end{lemma}
\begin{proof}
  Since $\WLPO$ is a proposition, we may assume to have $f:\RD\to\bool$
  and $x,y:\RD$ with $f(x)\neq f(y)$.  Let $P:\N\to\DProp$ be a
  decidable predicate.

  We start by setting up a decision procedure.  We define two
  sequences $a,b:\N\to\RD$ with $f(a_i)=\false$ and $f(b_i)=\true$ for
  each $i$, and so that $a$ and $b$ converge to the same real $l$.

  Without loss of generality, assume $f(x)=\false$ and $f(y)=\true$,
  and set:
  \begin{align*}
    a_0&\coloneqq x&
    b_0&\coloneqq y\\
    a_{n+1}&\coloneqq
             \begin{cases}
               \frac{a_n+b_n}2&\text{if $f\left(\frac{a_n+b_n}2\right)=\false$}\\
               a_n&\text{otherwise}
             \end{cases}&
    b_{n+1}&\coloneqq
             \begin{cases}
               \frac{a_n+b_n}2&\text{if $f\left(\frac{a_n+b_n}2\right)=\true$}\\
               b_n&\text{otherwise}
             \end{cases}
  \end{align*}

  In words, with $a_n$ and $b_n$ defined, we decide the next point by
  considering $f$ evaluated at the midpoint $\frac{a_n+b_n}2$, and
  correspondingly updating one of the points.  The sequences converge
  to the same point $l$.  Without loss of generality, we have
  $f(a_n)=f(l)=\false$ and $f(b_n)=\true$ for all $n:\N$.

  We may now decide $\neg\ex{n:\N}P(n)$.  We first define a sequence
  $c:\N\to\RD$ as follows.  For a given $n:\N$, we decide if there is
  any $i<n$ for which $P(i)$ holds, and if so, we set $c_n=b_i$ for
  the least such $i$.  Otherwise, we set $c_n=l$.

  The sequence $c$ converges, giving a limit $m:\RD$.  Consider
  $f(m)$.

  If $f(m)=\false$, then $\neg\ex{n:\N}P(n)$, since if there did exist
  $n$ with $P(n)$, then $m=b_i$ for some $i\leq n$, so that
  $f(m)=f(b_i)=\true$.

  If $f(m)=\true$, then $\neg\neg\ex{n:\N}P(n)$, since if
  $\fa{n:\N}\neg P(n)$ then $m=l$ and so $f(m)=\false$.
\end{proof}

The following key theorem explains the relationships between being a
Dedekind cut, having the Dedekind data $\cutStruct(x)$, and equipping
a real with a locator.
\begin{theorem}\label{thm:dedekind:struct:prop:rel}
  For a pair $x=(L,U)$ of predicates on the rationals we have the
  following:
  \begin{enumerate}
  \item $\cutStruct(x)\to\isCut(x)$,
  \item $\trunc{\cutStruct(x)}\Implies\isCut(x)$,
  \item $\isCut(x)\times\SL(x)\to\cutStruct(x)$,
  \item $\isCut(x)\times\trunc{\SL(x)}\Implies\isCauchyReal(x)$, and
  \item $\trunc{\cutStruct(x)}\Implies\isCauchyReal(x)$.
  \end{enumerate}
\end{theorem}
The third item tells us that for a given Dedekind real $x$, in order
to obtain the structures that make up $\cutStruct(x)$, we only require
$\SL(x)$.
\begin{proof}
  We show the first item by considering all property/structure-pairs
  above.

  $\bounderLower(x)\to\boundedLower(x)$ follows by applying the
  truncation map $\truncm{\,\cdot\,}$ of
  Definition~\ref{def:truncation}, and similarly for $\boundedUpper$.

  $\closerLower(x)\to\closedLower(x)$ is trivial since, following
  Definition~\ref{def:props:interp}, their definitions work out to the
  same thing: we do not need to make any changes to make
  $\closerLower$ structural.

  $\openerLower(x)\to\openLower(x)$ by a pointwise truncation: let
  $q:\Q$ be arbitrary and assume $q<x$, then we get $\sm{q' : \Q} (q <
  q') \times (q'<x)$, and hence $\ex{q' : \Q} {(q < q')} \land {(q'<x)}$.

  Again following Definition~\ref{def:props:interp}, $\transitive(x)$
  and $\transitor(x)$ are defined equally.

  $\SL(x)\to\located(x)$ again by a pointwise truncation.

  The second item follows using the elimination rule for propositional
  truncations since $\isCut(x)$ is a proposition.

  For the third item, it remains to construct bounds, and to construct
  $\openerLower(x)$ and $\openerUpper(x)$.  The former is
  Lemma~\ref{lem:bounders}.  The latter follows from the Archimedean
  structure of Lemma~\ref{lem:archimedean:struct} and the fact that we
  have locators for rationals, as in Lemma~\ref{lem:rationals}.

  The fourth item follows from
  Theorem~\ref{thm:locator:characterization}.

  The fifth item follows by combining the second and the fourth.
\end{proof}

\begin{theorem}
  For an arbitrary pair $x=(L,U)$ of predicates on the rationals it is
  not provable that $\isCut(x)$ implies $\trunc{\cutStruct(x)}$.
\end{theorem}
\begin{proof}
  By Theorem~\ref{thm:dedekind:struct:prop:rel},
  $\trunc{\cutStruct(x)}$ implies that $x$ is a Cauchy real.  However,
  in general the Cauchy reals and the Dedekind reals do not
  coincide~\cite{lubarsky:cauchy}.
\end{proof}

\section{Some constructive analysis with locators}
\label{sec:analysis}

We show some ways of using locators in an existing theory of
constructive analysis.  We re-emphasize that although the technique of
equipping numbers with locators can be applied to any archimedean
ordered field, for clarity and brevity we will work with the Dedekind
reals $\RD$, with more general description given in
Booij~\cite{booij:thesis}.

The central notion is that of functions on the reals that \emph{lift
  to locators}, discussed in Section~\ref{sec:analysis:preliminaries},
which is neither weaker nor stronger than continuity.  We compute
locators for integrals in Section~\ref{sec:integrals}.  We discuss how
locators can help computing roots of functions in
Section~\ref{sec:ivt}.

\subsection{Preliminaries}
\label{sec:analysis:preliminaries}

What are the functions on the reals that allow us to compute?  When
such a function $f:\RD\to\RD$ is applied to an input real number $x:\RD$
that we can compute with, then we should be able to compute with the
output $f(x)$.  This can be formalized in terms of locators in the
following straightforward way, which we use as an abstract notion of
computation.

\begin{definition}\label{def:lift:locs}
  A function $f:\RD\to\RD$ \emph{lifts to locators} if it comes equipped
  with a method for constructing a locator for $f(x)$ from a locator
  for $x$.  This means that $f$ lifts to locators if it is equipped with
  an element of the type
  \[
    \dpt{x:\RD}\SL(x)\to\SL(f(x)).
  \]

  Another way to say this is that $f$ lifts to locators iff we can find
  the top edge in the diagram
  \begin{center}
    \begin{tikzcd}
      \RDL \arrow[r] \arrow[d,"\fst"] \arrow[dr,phantom,"\circ"] &
      \RDL \arrow[d,"\fst"]
      \\
      \RD \arrow[r,"f"] & \RD
    \end{tikzcd}
  \end{center}
  where $\RDL\coloneqq\sm{x:\RD}\SL(x)$ is the type of real
  numbers equipped with locators.

  ``Lifting to locators'' itself is structure.
\end{definition}
\begin{remark}
  If the reals are defined intensionally, for example as the
  collection of all Cauchy sequences without quotienting, then every
  function on them is defined completely by its behavior on those
  intensional reals.  However, in our case, given only the lifting
  structure $\RDL\to\RDL$, we cannot recover the function $f:\RD\to\RD$,
  because we do not have a locator for every $x:\RD$.

  In other words, well-behaved maps are specified by two pieces of
  data, namely a function $f:\RD\to\RD$ representing the extensional
  value of the function, and a map $\RDL\to\RDL$ that tells us how to
  compute.
\end{remark}

\begin{example}\label{ex:exp:lifts}
  The exponential function $\exp(x)=\sum_{k=0}^\infty\frac{x^k}{k!}$
  of Examples~\ref{ex:exp} and~\ref{ex:exp:limit} lifts to locators, for
  example using our construction of locators for limits as in
  Lemma~\ref{lem:lim:locator}.
\end{example}

In order to start developing analysis, we define some notions of
continuity.
\begin{definition}\label{def:cty}
  A function $f:\RD\to\RD$ is \emph{continuous at $x:\RD$} if
  \[
    \fa{\varepsilon:\Q_+} \ex{\delta:\Q_+} \fa{y:\RD} \abs{x-y}
    < \delta\Implies\abs{f(x)-f(y)}<\varepsilon.
  \]
  $f$ is \emph{pointwise continuous} if it is continuous at all
  $x:\RD$.
\end{definition}

\begin{definition}\label{def:uniform:cty}
  A \emph{modulus of uniform continuity for $f$ on $[a,b]$}, with
  $a,b:\RD$, is a map $\omega:\Q_+\to\Q_+$ with:
  \[
    \fa{x,y\in[a,b]} \abs{x-y}<\omega(\varepsilon) \Implies
    \abs{f(x)-f(y)}<\varepsilon.
  \]
\end{definition}
\begin{example}[Continuity of $\exp$]\label{ex:exp:cty}
  For any $a,b$, there \emph{exists} a modulus of uniform continuity
  for $\exp$ on the range $[a,b]$.  If $a$ and $b$ have locators, then
  we can \emph{find} a modulus of uniform continuity for $\exp$ on
  that interval.
\end{example}

From a constructive viewpoint in which computation and continuity
align, it would be desirable if some form of continuity of $f:\RD\to\RD$
would imply that it lifts to locators.  Alas, this is not the case, not
even for constant functions.
\begin{lemma}
  If it holds that all constant functions lift to locators, then every
  $x:\RD$ comes equipped with a locator.
\end{lemma}
Using Lemmas~\ref{lem:magic:struct:taboo} and~\ref{lem:disc:wlpo},
this then yields the constructive taboo $\WLPO$.
\begin{proof}
  For $x:\RD$, let $f:\RD\to\RD$ be the constant map at $x$, and note
  that $f$ is continuous, so that by assumption it lifts to locators.
  Since the rational number $0$ has a locator, $f(0)=x$ has a locator.
\end{proof}
The converse direction, that lifting to locators would imply continuity,
also fails dramatically.
\begin{lemma}
  Assuming $\PEM$, we can define a discontinuous map $f:\RD\to\RD$ that
  lifts to locators.
\end{lemma}
\begin{proof}
  We can use $\PEM$ to define a discontinuous function, which
  automatically lifts to locators by applying Lemma~\ref{lem:locator:pem}.
\end{proof}
It may be the case that the structure of lifting to locators can be used
to strengthen certain \emph{properties} of continuity into
\emph{structures}.  For example, does every function that lifts to
locators and is pointwise continuous come equipped with the structure
\[
  \dpt{x:\RD}\dpt{\varepsilon:\Q_+} \sm{\delta:\Q_+} \fa{y:\RD} \abs{x-y} <
  \delta\Implies\abs{f(x)-f(y)}<\varepsilon
\]
of structural pointwise continuity at every $x:\RD$?  We leave this as
an open question.

\medskip

For the above reasons, the theorems in this section and the next
assume continuity \emph{and} a structure of lifting to locators: the
former to make the constructive analysis work, and the latter to
compute.

\subsection{Integrals}
\label{sec:integrals}

We can compute definite integrals of uniformly continuous functions in
the following way.

\begin{theorem}\label{thm:int:locator}
  Suppose $f:\RD\to\RD$ has a modulus of uniform continuity on
  $[a,b]$, and $a$ and $b$ are real numbers with locators.  Suppose
  that $f$ lifts to locators. Then $\int_a^bf(x)\dif x$ has a locator.
\end{theorem}
\begin{proof}
  For uniformly continuous functions, the integral
  $\int_a^bf(x)\dif x$ can be computed as the limit
  \[
    \lim_{n\to\infty}\frac{b-a}{n}\sum_{k=0}^{n-1}f\left(a+k\cdot\frac{b-a}{n}\right).
  \]
  Now every value
  \[
    \frac{b-a}{n}\sum_{k=0}^{n-1}f\left(a+k\cdot\frac{b-a}{n}\right).
  \]
  in the sequence comes equipped with a locator using
  Lemmas~\ref{lem:rationals} and~\ref{thm:ops}, and using the fact
  that $a$ and $b$ have locators and $f$ lifts to locators.  From the
  modulus of uniform continuity of $f$, and the computation of a
  rational $B$ with $b-a\leq B$ using Lemmas~\ref{thm:ops}
  and~\ref{lem:bounders} we can compute a modulus of Cauchy
  convergence of the sequence.  Hence the limit has a locator using
  Lemma~\ref{lem:lim:locator}.
\end{proof}

Combining this with the calculation of signed-digit representations of
reals with locators in Theorem~\ref{thm:sdr}, the above means we can
generate the digit sequence of certain integrals.  Through the
construction of close bounds in Lemma~\ref{lem:precise}, we can in
principle verify the value of integrals up to arbitrary precision.

\begin{remark}
  Integrals, as elements of $\RD$, can be defined given only the
  \emph{existence} of a modulus of uniform continuity.  To get a
  locator, we use the modulus of uniform continuity to find a modulus
  of Cauchy convergence.
\end{remark}
\begin{example}
  The integral $\int_0^8\sin(x+\exp(x))\dif x$ has a locator (where
  $\sin$ is defined, and shown to lift to locators, in a way similar to
  $\exp$).  This integral is often incorrectly approximated by
  computer algebra systems.  Mahboubi et
  al.~\cite[Section~6.1]{mahboubi:integrals} have formally verified
  approximations of this integral, and in principle our work gives an
  alternative method to do so.  However, our constructions are not
  efficient enough to do so in practice, and we give some possible
  remedies in the conclusions in Section~\ref{sec:conclusions}.
\end{example}

\subsection{Intermediate value theorems}
\label{sec:ivt}

We may often compute locators of real numbers simply by analysing the
proof of existing theorems in constructive analysis.  The following
construction of the root of a function is an example of us being able
to construct locators simply by following the proof in the literature.
\begin{theorem}
  Suppose $f$ is pointwise continuous on the interval $[a,b]$ and
  $f(a)<0<f(b)$ with $a,b:\RD$.  Then for every $\varepsilon:\Q_+$ we
  can find $x:\RD$ with $\abs{f(x)}<\varepsilon$.  If $f$ lifts to
  locators, and $a$ and $b$ are equipped with locators, then $x$ is
  equipped with a locator.
\end{theorem}
\begin{proof}
  The first claim is shown as in Frank~\cite{frank:aivt} by defining
  sequences $c,d,z,w:\N\to\RD$:
  \begin{align*}
    z_0&=a&c_n&=(z_n+w_n)/2&z_{n+1}=c_n-d_n(b-a)/2^{n+1}\\
    w_0&=b&d_n&=\max\left(0,\min\left(\frac12+\frac{f(c_n)}{\varepsilon},1\right)\right)&w_{n+1}=w_n-d_n(b-a)/2^{n+1}
  \end{align*}
  with $x$ defined as the limit of $c:\N\to\RD$, which converges since
  $z,w:\N\to\RD$ are monotone sequences with $z_n\leq c_n\leq w_n$ and
  $z_n-w_n=(b-a)/2^{n}$.  Because $f$ lifts to locators, and $a$ and
  $b$ have a locator, all $c_n$ have locators.  For a modulus of
  Cauchy convergence, Lemma~\ref{thm:ops} gives a locator for $b-a$ so
  that we can use Lemma~\ref{lem:bounders} to compute a rational $B$
  with $\abs{z_n-w_n}\leq B/2^n$.  So by
  Lemma~\ref{lem:lim:locator}, $x$ has a locator.
\end{proof}

We will now work towards an intermediate value theorem in which the
locators help us with the computation of the root itself, avoiding any
choice principles.  We stated this intermediate value theorem and its
proof informally in the introduction to
Section~\ref{sec:locators}.
\begin{definition}
  A function $f:\RD\to\RD$ is \emph{locally nonconstant} if for all
  $x<y$ and $t:\RD$, there exists $z:\RD$ with $x<z<y$ and
  $f(z)\apart t$, recalling that
  $(f(z)\apart t)=(f(z)>t)\vee (f(z)<t)$.
\end{definition}
\begin{example}
  Every strictly monotone function is locally nonconstant, but not
  every locally nonconstant function is strictly monotone.
\end{example}
\begin{lemma}\label{lem:str:nonc}
  Suppose $f$ is a pointwise continuous function, and $x$, $y$ and $t$
  are real numbers with locators with $x<y$.  Further suppose that $f$ is locally
  nonconstant, and lifts to locators.
  Then we can find $r:\Q$ with $x<r<y$ and $f(r)\apart t$.
\end{lemma}
\begin{proof}
  Since $f$ is locally nonconstant, there \emph{exist} $z:\RD$ and
  $\varepsilon:\Q_+$ with $\abs{f(z)-t}>\varepsilon$.  Since $f$ is
  continuous at $z$, there exists $q:\Q$ with
  $\abs{f(q)-t}>\varepsilon/2$.  Since $\Q_+$ and $\Q$ are
  denumerable, we can find $r:\Q$ such that there exists
  $\eta:\Q_+$ with $\abs{f(r)-t}>\eta$.  In particular $r$ satisfies
  $\abs{f(r)-t}>0$, that is, $f(r)\apart t$.
\end{proof}
The above result can be thought of as saying that if $f$ is a
pointwise continuous function that lifts to locators, then the
\emph{property} of local nonconstancy implies a certain
\emph{structure} of local nonconstancy: for given reals with locators
$x<y$ and $t$, we do not just get the existence of a real $z$, but we
can explicitly choose a point $z$ where $f$ is apart from $t$.

Exact intermediate value theorems based on local nonconstancy usually
assume dependent choice, see e.g.\ Bridges and
Richman~\cite[Chapter~3, Theorem~2.5]{BridgesRichman:Varieties} or
Troelstra and van Dalen~\cite[Chapter~6,
Theorem~1.5]{Troelstra:vanDalen:1}.  The following result holds in the
absence of such choice principles.  It can perhaps be compared to
developments in which the real numbers are represented directly as
Cauchy
sequences~\cite{Schuster03constructivesolutions,schwichtenberg:witnesses,DBLP:journals/apal/Hendtlass12}
or with Taylor~\cite{taylor:lamcra}.  Note, however, that
\begin{enumerate}
\item we assume local nonconstancy rather than monotonicity, and that
\item we use the \emph{property} of local nonconstancy to compute
  roots, rather than assuming this as structure.
\end{enumerate}
\begin{theorem}\label{thm:exact:ivt}
  Suppose $f$ is a pointwise continuous function, and $a<b$ are real
  numbers with locators.  Further suppose that $f$ is locally
  nonconstant, and lifts to locators, with $f(a)\leq0\leq f(b)$.  Then we
  can find a root of $f$, which comes equipped with a locator.
\end{theorem}
\begin{proof}
  We define sequences $a,b:\N\to\RD$ with $a_n<a_{n+1}<b_{n+1}<b_n$,
  with $f(a_n)\leq0\leq f(b_n)$, with
  $b_n-a_n\leq(b-a)\left(\frac{2}{3}\right)^n$, and such that all
  $a_n$ and $b_n$ have locators.  Set $a_0=a$, $b_0=b$.  Suppose $a_n$
  and $b_n$ are defined, and use Lemma~\ref{lem:str:nonc} to find
  $q_n$ with $\frac{2a_n+b_n}{3}<q_n<\frac{a_n+2b_n}{3}$ and
  $f(q_n)\apart 0$.
  \begin{itemize}
  \item If $f(q_n)>0$, then set $a_{n+1}\coloneqq a_n$
    and $b_{n+1}\coloneqq q_n$.
  \item If $f(q_n)<0$, then set $a_{n+1}\coloneqq q_n$ and
    $b_{n+1}\coloneqq b_n$.
  \end{itemize}
  For a modulus of Cauchy convergence, we can compute a locator for
  $b-a$ and from this we can compute a rational $B$ with
  $\abs{b_n-a_n}\leq B\left(\frac{2}{3}\right)^n$.  The sequences
  converge to a number $x$.  For any $\varepsilon$, we have
  $\abs{f(x)}\leq\varepsilon$, hence $f(x)=0$.
\end{proof}

\begin{remark}
  Since we only appealed to Lemma~\ref{lem:str:nonc} with $t=0$, that
  is, since we were only interested in points where $f$ is apart from
  0, Theorem~\ref{thm:exact:ivt} may be strengthened by only requiring
  that $f$ is locally nonzero.
\end{remark}

Theorem~\ref{thm:exact:ivt} is an improvement on existing exact
intermediate value
theorems~\cite{Schuster03constructivesolutions,taylor:lamcra} since it
assumes the \emph{property} of local nonconstancy to compute roots.

\begin{example}
  The function $\exp$ is strictly increasing, and hence locally
  nonconstant.  So if $y>0$ has a locator, then $\exp(x)=y$ has a
  solution $x$ with a locator.
\end{example}

\section{Closing remarks}
\label{sec:conclusions}
We have paid attention to the difference between property and
structure while defining the real numbers and other foundations of
constructive analysis.  We have introduced the term \emph{locator} to
mean the structure that is the focus of this paper, and have
introduced a basic theory of locators.  The fact that the results
about locators have equivalents in terms of intensional
representations of reals suggests that we are not doing anything new.
This is desirable: we merely introduced a particular representation
that seems suitable for computation.  The presence of the locators is
not to make the constructive analysis work; rather, it is to make the
computation work.  In this sense, we have made the computation work
without a conceptual burden of intensional representations.

The constructions and results remind of computable analysis.  But our
development is orthogonal to computability: even reals that are not
computable in some semantics can have locators, for example in the
presence of choice axioms, in which case all reals have locators.

Locators allow to observe information of real numbers, such as
signed-digit expansions.  We have shown the interdefinability of
locators with Cauchy sequences, and in this way we characterized the
Cauchy reals as those Dedekind reals for which a locator exists.

The new notion of \emph{lifting to locators} grew out of a naive desire
to have locators for the output of a function whenever we have a
locator for the input.  We have left the following open question:
given that $f:\RD\to\RD$ lifts to locators, do we obtain a certain
\emph{structure} of continuity from a \emph{property} of continuity?

We have not spent much time finding an alternative notion of
``functions that compute'' with a closer relationship to continuity,
and this could be the topic of further research.  Such a notion could
perhaps allow for more satisfying formulations of the theorems in
Sections~\ref{sec:integrals} and~\ref{sec:ivt}.

Our work allows to obtain signed-digit representations of integrals.
These results are based on backwards error propagation, essentially
due to our notion of lifting to locators.  The advantage of this is that
we are guaranteed to be able to find results.  However, forward error
propagation, as in Mahboubi et al.~\cite{mahboubi:integrals}, may be
more efficient.  It may be possible to combine the naturalness of
locators with forward error propagation by equipping the real numbers
involved with bounds as in the remark below Lemma~\ref{lem:bounders}.
Having shown that we can compute arbitrarily precise approximations to
reals with locators in Lemma~\ref{lem:precise}, we may as well equip
real numbers with an efficient method for doing so.  Thus, in future
work, some of the techniques of previous work on verified computation
with exact reals may be developed in our setting as well.

Another possible future direction is to find a more general notion of
locator that applies to more general spaces, such as the complex
plane, function spaces, or metric spaces.  This could then be a
framework for observing information of differential equations, which
are also discussed in a more general description of
locators~\cite{booij:thesis}.

The work lends itself to being formalized in proof assistants such as
Agda or Coq.  In this way we can automatically obtain algorithms from
proofs.  Part of the work has indeed been formalized in
Coq~\cite{booij:coq:locators}.  Results in the work above correspond
with the formalized proofs in \verb|theories/Analysis/Locator.v| as
follows: Lemma~\ref{lem:locator:pem} as \verb|all_reals_locators|,
Lemma~\ref{lem:rationals} as \verb|locator_left| and
\verb|locator_right|, Lemma~\ref{lem:loc:dec} as
\verb|equiv_locator_locator'|, Lemma~\ref{lem:g:dec} as
\verb|nltqx_locates_left| and \verb|nltxr_locates_right|,
Lemma~\ref{lem:bounders} as \verb|lower_bound| and \verb|upper_bound|,
Lemma~\ref{lem:precise} as \verb|tight_bound|,
Lemma~\ref{lem:archimedean:struct} as \verb|archimedean_structure|,
and the majority of Theorem~\ref{thm:ops}, as well as
Lemma~\ref{lem:lim:locator}, as the terms starting with
\verb|locator_|.  This development has been merged
into the HoTT library~\cite{DBLP:conf/cpp/BauerGLSSS17}.
But we may worry that the
proofs we provided are not sufficiently efficient for useful
calculations, and we intend to address this important issue in future
work.

\subsection*{Acknowledgements}
\label{sec:acknowledgements}

We would like to thank the reviewers for their invaluable feedback.
This project has received funding from the European Union's Horizon
2020 research and innovation programme under the Marie
Skłodowska-Curie grant agreement No 731143.

\bibliography{analysis-utt}

\begin{thebibliography}{10}

\bibitem{DBLP:conf/cpp/BauerGLSSS17}
Andrej Bauer, Jason Gross, Peter~LeFanu Lumsdaine, Michael Shulman, Matthieu
  Sozeau, and Bas Spitters.
\newblock The hott library: a formalization of homotopy type theory in coq.
\newblock In {\em Proceedings of the 6th {ACM} {SIGPLAN} Conference on
  Certified Programs and Proofs, {CPP} 2017, Paris, France, January 16-17,
  2017}, pages 164--172, 2017.
\newblock \href {https://doi.org/10.1145/3018610.3018615}
  {\path{doi:10.1145/3018610.3018615}}.

\bibitem{bishop:constructive}
Errett Bishop and Douglas Bridges.
\newblock {\em Constructive analysis}.
\newblock Springer-Verlag, Berlin, Heidelberg, 1985.
\newblock \href {https://doi.org/10.1007/978-3-642-61667-9}
  {\path{doi:10.1007/978-3-642-61667-9}}.

\bibitem{booij:coq:locators}
Auke~Bart Booij.
\newblock Coq development of locators.
\newblock Accessed February 2020.
\newblock URL: \url{https://github.com/abooij/HoTT/tree/locators}.

\bibitem{booij:thesis}
Auke~Bart Booij.
\newblock {\em Analysis in Univalent Type Theory}.
\newblock PhD thesis, University of Birmingham, 2020.

\bibitem{BridgesRichman:Varieties}
Douglas Bridges and Fred Richman.
\newblock {\em Varieties of Constructive Mathematics}.
\newblock Number~97 in London Mathematical Society Lecture Notes. Cambridge
  University Press, 1987.

\bibitem{bridges:vita}
Douglas Bridges and Luminita~Simona Vita.
\newblock {\em Techniques of Constructive Analysis}.
\newblock Springer-Verlag, New York, 2006.
\newblock \href {https://doi.org/10.1007/978-0-387-38147-3}
  {\path{doi:10.1007/978-0-387-38147-3}}.

\bibitem{brouwer:zahl}
Luitzen~E.J. Brouwer.
\newblock Besitzt jede reelle zahl eine dezimalbruchentwicklung?
\newblock {\em Mathematische Annalen}, 83:201--210, 1921.
\newblock URL: \url{http://eudml.org/doc/158869}.

\bibitem{lcf:analysis}
Luis Cruz-Filipe.
\newblock {\em Constructive Real Analysis: a Type-Theoretical Formalization and
  Applications}.
\newblock PhD thesis, University of Nijmegen, April 2004.

\bibitem{escardo:nat:dec}
Mart{\'{\i}}n~H{\"{o}}tzel Escard{\'{o}}.
\newblock Message to the {Univalent Foundations} mailing list.
\newblock
  \url{https://groups.google.com/d/msg/univalent-foundations/SA0dzenV1G4/d5iIGdKKNxMJ},
  March 2013.

\bibitem{escardo:xu:inconsistency}
Mart{\'{\i}}n~H{\"{o}}tzel Escard{\'{o}} and Chuangjie Xu.
\newblock The inconsistency of a brouwerian continuity principle with the
  curry-howard interpretation.
\newblock In {\em 13th International Conference on Typed Lambda Calculi and
  Applications, {TLCA} 2015, July 1-3, 2015, Warsaw, Poland}, pages 153--164,
  2015.
\newblock \href {https://doi.org/10.4230/LIPIcs.TLCA.2015.153}
  {\path{doi:10.4230/LIPIcs.TLCA.2015.153}}.

\bibitem{frank:aivt}
Matthew Frank.
\newblock {Interpolating Between Choices for the Approximate Intermediate Value
  Theorem}.
\newblock {\em ArXiv e-prints}, January 2017.
\newblock \href {http://arxiv.org/abs/1701.02227} {\path{arXiv:1701.02227}}.

\bibitem{DBLP:journals/apal/Hendtlass12}
Matthew Hendtlass.
\newblock The intermediate value theorem in constructive mathematics without
  choice.
\newblock {\em Ann. Pure Appl. Logic}, 163(8):1050--1056, 2012.
\newblock \href {https://doi.org/10.1016/j.apal.2011.12.026}
  {\path{doi:10.1016/j.apal.2011.12.026}}.

\bibitem{hofmann:extensional}
Martin Hofmann.
\newblock {\em Extensional concepts in intensional type theory}.
\newblock PhD thesis, University of Edinburgh, 1995.

\bibitem{ishihara:sequentially}
Hajime Ishihara.
\newblock Sequentially continuity in constructive mathematics.
\newblock In C.~S. Calude, M.~J. Dinneen, and S.~Sburlan, editors, {\em
  Combinatorics, Computability and Logic}, pages 5--12, London, 2001. Springer
  London.
\newblock \href {https://doi.org/10.1007/978-1-4471-0717-0_2}
  {\path{doi:10.1007/978-1-4471-0717-0_2}}.

\bibitem{DBLP:journals/corr/KrausECA16}
Nicolai Kraus, Mart{\'{\i}}n~H{\"{o}}tzel Escard{\'{o}}, Thierry Coquand, and
  Thorsten Altenkirch.
\newblock Notions of anonymous existence in {M}artin-{L}{\"{o}}f type theory.
\newblock {\em Logical Methods in Computer Science}, 13(1), 2017.
\newblock \href {https://doi.org/10.23638/LMCS-13(1:15)2017}
  {\path{doi:10.23638/LMCS-13(1:15)2017}}.

\bibitem{lubarsky:cauchy}
Robert~S. Lubarsky.
\newblock On the cauchy completeness of the constructive cauchy reals.
\newblock {\em Electr. Notes Theor. Comput. Sci.}, 167:225--254, 2007.
\newblock \href {https://doi.org/10.1016/j.entcs.2006.09.012}
  {\path{doi:10.1016/j.entcs.2006.09.012}}.

\bibitem{mahboubi:integrals}
Assia Mahboubi, Guillaume Melquiond, and Thomas Sibut{-}Pinote.
\newblock Formally verified approximations of definite integrals.
\newblock In {\em Interactive Theorem Proving--7th International Conference,
  {ITP} 2016, Nancy, France, August 22-25, 2016, Proceedings}, pages 274--289,
  2016.
\newblock \href {https://doi.org/10.1007/978-3-319-43144-4_17}
  {\path{doi:10.1007/978-3-319-43144-4_17}}.

\bibitem{oconnor:formalizing}
Russell O'Connor.
\newblock {\em Incompleteness \& Completeness: Formalizing Logic and Analysis
  in Type Theory}.
\newblock PhD thesis, Radboud Universiteit Nijmegen, 2009.

\bibitem{Schuster03constructivesolutions}
Peter Schuster and Helmut Schwichtenberg.
\newblock Constructive solutions of continuous equations.
\newblock In Godehard Link, editor, {\em de Gruyter Series in Logic and Its
  Applications}. Walter de Gruyter, January 2004.
\newblock \href {https://doi.org/10.1515/9783110199680.227}
  {\path{doi:10.1515/9783110199680.227}}.

\bibitem{schwichtenberg:witnesses}
Helmut Schwichtenberg.
\newblock Constructive analysis with witnesses.
\newblock January 2017.
\newblock URL:
  \url{http://www.math.lmu.de/~schwicht/seminars/semws16/constr16.pdf}.

\bibitem{taylor:lamcra}
Paul Taylor.
\newblock A lambda calculus for real analysis.
\newblock {\em J. Logic {\&} Analysis}, 2, 2010.
\newblock URL:
  \url{http://logicandanalysis.org/index.php/jla/article/view/63/25}.

\bibitem{hottbook}
{The Univalent Foundations Program}.
\newblock {\em Homotopy Type Theory: Univalent Foundations of Mathematics}.
\newblock Institute for Advanced Study, 2013.
\newblock URL: \url{https://homotopytypetheory.org/book}.

\bibitem{Troelstra:vanDalen:1}
Anne~S. Troelstra and Dirk {van Dalen}.
\newblock {\em Constructivism in Mathematics, an Introduction}, volume~1 of
  {\em Studies in Logic and the Foundations of Mathematics}.
\newblock North-Holland, August 1988.

\bibitem{Turing:1936:CNA}
Alan~M. Turing.
\newblock On computable numbers, with an application to the
  {Entscheidungsproblem}.
\newblock {\em Proceedings of the London Mathematical Society. Second Series},
  42:230--265, 1936.
\newblock See correction \cite{turing:correction}.

\bibitem{turing:correction}
Alan~M. Turing.
\newblock On computable numbers, with an application to the
  {Entscheidungsproblem}. {A} correction.
\newblock {\em Proceedings of the London Mathematical Society. Second Series},
  43:544--546, 1937.
\newblock See \cite{Turing:1936:CNA}.

\bibitem{wiedmer1977exaktes}
Edwin Wiedmer.
\newblock {\em Exaktes Rechnen mit reellen Zahlen und anderen unendlichen
  Objekten}.
\newblock PhD thesis, ETH Zurich, 1977.

\end{thebibliography}
\clearpage
\appendix

\end{document}